\documentclass[12pt]{amsart}
\usepackage{amscd,amsmath,amsthm,amssymb}

\usepackage[pagewise]{lineno}

\usepackage{subfig}
\usepackage{amsmath, amsthm, amscd, amsfonts, amssymb, enumerate, verbatim, newlfont, calc, graphicx, color}
\usepackage[bookmarksnumbered, colorlinks=false, plainpages]{hyperref}
\usepackage{tikz}
\usepackage{doi}
\usepackage{graphicx}
\usepackage[margin=2.5cm]{geometry}
\usepackage{floatflt,epsfig}
\usepackage{colortbl}
\usepackage{xcolor}
\usepackage{array}
\newcolumntype{H}{>{\setbox0=\hbox\bgroup}c<{\egroup}@{}}
\newcolumntype{Z}{>{\setbox0=\hbox\bgroup}c<{\egroup}@{\hspace*{-\tabcolsep}}}

\usepackage{pstcol,pst-plot,pst-3d}
\usepackage{lmodern,pst-node}
\usepackage[normalem]{ulem}
\usepackage{hyperref}
\psset{unit=0.7cm,linewidth=0.8pt,arrowsize=2.5pt 4}

\newpsstyle{fatline}{linewidth=1.5pt}
\newpsstyle{fyp}{fillstyle=solid,fillcolor=verylight}
\definecolor{verylight}{gray}{0.97}
\definecolor{light}{gray}{0.93}
\definecolor{medium}{gray}{0.82}

\def\NZQ{\Bbb}               
\def\NN{{\NZQ N}}

\def\ZZ{{\NZQ Z}}

\def\F{{\mathcal F}}

\def\opn#1#2{\def#1{\operatorname{#2}}} 
\opn\chara{char} \opn\length{\ell} \opn\pd{pd} \opn\rk{rk}
\opn\projdim{proj\,dim} \opn\injdim{inj\,dim} \opn\rank{rank}
\opn\depth{depth} \opn\grade{grade} \opn\height{height}
\opn\embdim{emb\,dim} \opn\codim{codim}

\opn\Tr{Tr} \opn\bigrank{big\,rank}
\opn\superheight{superheight}\opn\lcm{lcm}
\opn\trdeg{tr\,deg}
	\opn\reg{reg} \opn\lreg{lreg} \opn\ini{in} \opn\lpd{lpd}
	\opn\size{size} \opn\sdepth{sdepth}
	\opn\link{link}\opn\fdepth{fdepth}\opn\lex{lex}\opn\dist{dist}
	\opn\div{div} \opn\Div{Div} \opn\cl{cl} \opn\Cl{Cl}
	\opn\Spec{Spec} \opn\Supp{Supp} \opn\supp{supp} \opn\Sing{Sing}
	\opn\Ass{Ass} \opn\Min{Min}\opn\Mon{Mon}
	\opn\Ann{Ann} \opn\Rad{Rad} \opn\Soc{Soc}
	\opn\Im{Im} \opn\Ker{Ker} \opn\Coker{Coker} \opn\Am{Am}
	\opn\Hom{Hom} \opn\Tor{Tor} \opn\Ext{Ext} \opn\End{End}
	\opn\Aut{Aut} \opn\id{id}
	
	\opn\nat{nat}
	\opn\pff{pf}
	\opn\Pf{Pf} \opn\GL{GL} \opn\SL{SL} \opn\mod{mod} \opn\ord{ord}
	\opn\Gin{Gin} \opn\Hilb{Hilb}\opn\sort{sort}
	\opn\aff{aff} \opn
	\con{conv} \opn\relint{relint} \opn\st{st}
	\opn\lk{lk} \opn\cn{cn} \opn\core{core} \opn\vol{vol}
	\opn\link{link} \opn\star{star}\opn\lex{lex}\opn\set{set}
	\opn\gr{gr}
	
	\def\pot#1#2{#1[\kern-0.28ex[#2]\kern-0.28ex]}

	\opn\dirlim{\underrightarrow{\lim}}
	\opn\inivlim{\underleftarrow{\lim}}
	\def\Implies{\ifmmode\Longrightarrow \else
		\unskip${}\Longrightarrow{}$\ignorespaces\fi}
	\def\implies{\ifmmode\Rightarrow \else
		\unskip${}\Rightarrow{}$\ignorespaces\fi}
	\def\iff{\ifmmode\Longleftrightarrow \else
		\unskip${}\Longleftrightarrow{}$\ignorespaces\fi}

	\let\:=\colon
	\newtheorem{Theorem}{Theorem}[section]
	\newtheorem{Lemma}[Theorem]{Lemma}
	\newtheorem{Corollary}[Theorem]{Corollary}
	\newtheorem{Proposition}[Theorem]{Proposition}
	\newtheorem{Remark}[Theorem]{Remark}
	
	\newtheorem{Example}[Theorem]{Example}

	\newtheorem{Conjecture}[Theorem]{Conjecture}
	\newtheorem{Question}[Theorem]{Question}
    
    	\newtheorem{Notation}[Theorem]{Notation}

	\let\epsilon\varepsilon
	\let\kappa=\varkappa

	\def\qed{\ifhmode\textqed\fi
		\ifmmode\ifinner\quad\qedsymbol\else\dispqed\fi\fi}
	\def\textqed{\unskip\nobreak\penalty50
		\hskip2em\hbox{}\nobreak\hfil\qedsymbol
		\parfillskip=0pt \finalhyphendemerits=0}
	\def\dispqed{\rlap{\qquad\qedsymbol}}
	
	\opn\dis{dis}
	\def\pnt{{\raise0.5mm\hbox{\large\bf.}}}
	
	\opn\Lex{Lex}
	
\begin{document}	 
        \title[] {On squarefree powers of simplicial trees}
	
        \author {Elshani Kamberi}
        \address{Sabanci University, Faculty of Engineering and Natural Sciences, Orta Mahalle, Tuzla 34956, Istanbul, Turkey}
	\email{elshani@sabanciuniv.edu}

        \author{Francesco Navarra}
	\address{Sabanci University, Faculty of Engineering and Natural Sciences, Orta Mahalle, Tuzla 34956, Istanbul, Turkey}
	\email{francesco.navarra@sabanciuniv.edu}
        \author{Ayesha Asloob Qureshi}
        \address{Sabanci University, Faculty of Engineering and Natural Sciences, Orta Mahalle, Tuzla 34956, Istanbul, Turkey}
	\email{aqureshi@sabanciuniv.edu, ayesha.asloob@sabanciuniv.edu}
		
	
		\thanks{Authors are supported by Scientific and Technological Research Council of Turkey T\"UB\.{I}TAK under the Grant No: 122F128, and are thankful to T\"UB\.{I}TAK for their supports. The last author also acknowledge support from the ICTP through the Associates Programme (2019-2024).}  
        \maketitle
            \begin{abstract}    
In this article, we study the squarefree powers of facet ideals associated with simplicial trees. Specifically, we examine the linearity of their minimal free resolution and their regularity. Additionally, we investigate when the first syzygy module of squarefree powers of facet ideal of a simplicial tree is generated by linear relations. Finally, we provide a combinatorial formula for the regularity of the squarefree powers of $t$-path ideals of path graphs.
 \end{abstract}
	
     \section*{Introduction}
    In the last two decades, the study of the regularity of powers of squarefree monomial ideals has become a notable trend in combinatorial commutative algebra. This line of research took one of its initial steps with a beautiful theorem (see \cite{CHT, K}), which states that if $I$ is a homogeneous ideal in a polynomial ring, then the function $\reg(I^k)$ is asymptotically linear. This result gathered the interest of many algebraists to investigate the regularity and, more generally, the minimal graded free resolution of the powers of homogeneous ideals, leading to the publication of many interesting papers on this topic. However, obtaining a complete understanding of the asymptotic linearity of the regularity function for arbitrary homogeneous ideals seems almost impossible, and this problem remains open even for squarefree monomial ideals. 
    
    In this context, the study of the regularity of the squarefree powers of squarefree monomial ideals assumes an important role. Let $I$ be a squarefree monomial ideal. The $k$-th squarefree power $I^{[k]}$ of  $I$ is the ideal generated by the squarefree generators of $I^k$. The squarefree powers of $I$ provide important information on the ordinary powers of $I$. It is known from \cite[Lemma 4.4]{HHZ} that the multigraded minimal free resolution $I^{[k]}$ is a subcomplex of the multigraded minimal free resolution of $I^k$. Consequently, $\reg I^{[k]} \leq \reg I^k$, and if $I^{[k]}$ does not have a linear resolution, then $I^k$ also does not have a linear resolution. Another interesting aspect of the study of squarefree powers comes from their deep link with the matching theory of simplicial complexes (or hypergraphs). Recall that $I$ can be viewed both as an edge ideal of a hypergraph and as a facet ideal of a simplicial complex. In this work, we will adopt the latter terminology. Let $\Delta$ be a simplicial complex and $I(\Delta)$ be the facet ideal of $\Delta$.
     A matching of $\Delta$ is a set of pairwise disjoint facets of $\Delta$. Indeed, the generators of $I(\Delta)^{[k]}$ correspond to the matching of $\Delta$ of size $k$. This means that $I(\Delta)^{[k]}\neq 0$ only when $1\leq k\leq \nu(\Delta)$, where $\nu(\Delta)$ is the maximum size of a matching of $\Delta$. The study of squarefree powers began with \cite{BHZ} for facet ideals of 1-dimensional simplicial complexes, or simply, the edge ideals of graphs. Since then, several papers, including \cite{CFL, EF, EHHM, EHHM2, EH, FHH, S}, have been published on this topic, focusing on the squarefree powers of edge ideals of different classes of graphs.

Inspired by this, in this article we study the homological properties of the squarefree powers of squarefree monomial ideals that are not necessarily quadratic. To this end, we begin our investigation with the squarefree monomial ideals attached to simplicial trees. In \cite{F1}, Faridi introduced simplicial trees as a natural generalization of trees in the context of graphs. In the language of hypergraphs, simplicial trees correspond to totally balanced hypergraphs, \cite[Theorem 3.2]{HHTZ}, that is the hypergraphs without any ``special" cycles. We direct reader to \cite[Chapter 5]{B}, for more details.

  A breakdown of the contents of this paper is as follows: in Section~\ref{Section one: Preliminaries}, we recall necessary definitions and notions related to simplicial complexes and establish some preliminary results concerning simplicial trees. Section~\ref{Section: Simplicial forests with linear squarefree powers} is devoted to studying the linearity of the resolutions of the squarefree powers of the facet ideals of simplicial trees. In \cite[Theorem 3.17]{Z}, Zheng showed that a facet ideal of a simplicial tree $\Delta$ has a linear resolution if and only if $\Delta$ satisfies the so-called intersection property. More recently, the authors of \cite{KK1} proved that if the facet ideal $I(\Delta)$ of a simplicial tree has a linear resolution, then $I(\Delta)^k$ has a linear resolution for all $k \geq 1$. A natural question in this context is whether the squarefree powers of $I(\Delta)$ also admit the same property, that is, if $I(\Delta)$ has a linear resolution, then does $I(\Delta)^{[k]}$ also have a linear resolution? However, in Proposition~\ref{lem:matchingnumber} we prove that if $I(\Delta)$ has a linear resolution, equivalently if $\Delta$ has the intersection property, then $\nu(\Delta) \leq 2$. This implies that $I(\Delta)^{[k]} = 0$ for all $k > 2$, meaning the only squarefree power to be considered is $I(\Delta)^{[2]}$. The linearity of the resolution of $I(\Delta)^{[2]}$ follows from \cite[Proposition 2.10]{JZ} together with \cite[Lemma 2.3]{KK1}, as discussed in Theorem~\ref{thm: simplicial tree with i.p. has linear quotients}.

\medskip

   The next question that we tackle in Section~\ref{Section: Simplicial forests with linear squarefree powers} is motivated by \cite[Theorem 5.1]{BHZ}, in which the authors show that the highest non-vanishing power of an edge ideal $I(G)$ of any graph $G$ admits a linear resolution. Theorem~\ref{thm: simplicial tree with i.p. has linear quotients} also implies that the highest non-vanishing squarefree power of the facet ideal of a simplicial tree with the intersection property has a linear resolution. However, this is not true in general for simplicial trees, as observed in Example~\ref{exa: 3-path of a rooted tree with no linear resolution} for a suitable $3$-path ideal of a rooted tree graph. The $t$-path ideals of graphs were introduced in \cite{CDN} as ideals generated by monomials that correspond to paths of length $t-1$ of a graph $G$. If $G$ is a directed graph, then one may consider the $t$-path ideal of $G$ as an ideal generated by monomials corresponding to the directed paths of length $t-1$. For $t=2$, the $t$-path ideal coincides with the edge ideal of $G$. In \cite{JV}, the authors proved that the $t$-path ideal of a rooted tree (a special class of directed trees) is the facet ideal of a simplicial tree, providing a rich class of simplicial trees. The $t$-path ideal of a rooted tree have been studied by many authors, for example see \cite{AF, AF2, BHaK, KK1}.  Given a rooted tree $\Gamma$, we denote the simplicial complex whose facets are $t$-paths of $\Gamma$ as $\Gamma_t$. As observed in Example~\ref{exa: 3-path of a rooted tree with no linear resolution}, the highest non-vanishing power of $I(\Gamma_t)$ need not have a linear resolution. This motivated us to state Theorem~\ref{thm:broomgraph-linearquotient}, where we show that the highest non-vanishing squarefree power of $I(\Gamma_t)$ has a linear resolution if $\Gamma_t$ is a broom graph.

  Section~\ref{Section: linearly related property} is devoted to understanding when the first syzygy module of the squarefree powers of the facet ideals of simplicial forests is generated by linear forms. It turns out that the restricted matching of a simplicial complex $\Delta$ plays an important role in this context. The restricted matching of a graph (1-dimensional simplicial complex) was introduced in \cite{BHZ}, and we extend this definition for any $d$-dimensional simplicial complex; see Section~\ref{Section one: Preliminaries} for the formal definition. We prove in Theorem~\ref{thm: k-squarefree power non-linearly related} that if $\Delta$ is a simplicial forest, then the squarefree powers of $I(\Delta)$ are not linearly related up to the restricted matching number of $\Delta$. It is shown in \cite[Theorem 3.1]{EHHM} that given a graph $G$, if $I(G)^{[k]}$ is linearly related, then $I(G)^{[k+1]}$ is linearly related. We prove an analogue of this result in the case when $\Delta$ is a pure simplicial tree. This also shows that if the highest squarefree power of $I(\Delta)$ is not linearly related, then none of the non-vanishing squarefree powers have a linear resolution. We further investigate the first syzygy modules for a special class of simplicial trees, laying the groundwork for discussing regularity in the next section. Utilizing a celebrated result by Gasharov, Peeva, and Welker \cite{GPW}, we establish the following theorem:
  \medskip
  
   {\bf Theorem} (see Theorem~\ref{theorem: betti number t-path ideal})	{\it Let $\Gamma_{n,t}$ be the $t$-path simplicial tree of a path graph $P_n$ and $I_{n,t}= I(\Gamma_{n,t})$. Then $\beta_{1,p}\left(I_{n,t}^{[k]}\right) = 0$ if $p \notin \{ kt+1,(k+1)t\}$.}
\medskip

   Recently, in \cite[Section 4]{KK1}, the authors discussed the regularity of certain ordinary powers of $t$-path ideals of broom graphs and proposed a conjecture on the upper bound of the regularity of facet ideals of simplicial trees. It is noteworthy that path graphs constitute a special subclass of broom graphs. Motivated by their findings, in Section~\ref{Section: Regularity of t-path ideals of path graph}, we concentrate on the $t$-path ideals of path graphs. In Theorem~\ref{thm: equality regularity t-path ideals of path}, we present a combinatorial formula for their regularity.
 To achieve this, we utilize well-known exact sequences for homogeneous ideals (see Theorem~\ref{thm: upper bound}) and employ both combinatorial and topological techniques (see Theorem~\ref{thm: lower bound}). The main result of this section is summarized in the following theorem.

    \medskip
  
   {\bf Theorem} (see Theorem~\ref{thm: equality regularity t-path ideals of path})
    {\it Let $\Gamma_{n,t}$ be the $t$-path simplicial tree of a path graph $P_n$ and $I_{n,t}= I(\Gamma_{n,t})$. Then for any $1\leq k+1 \leq \nu(\Gamma_{n,t})$, we  have
		\[
		\reg\left( \frac{R}{I_{n,t}^{[k+1]}}\right)= kt+(t-1)\nu_1 (\Gamma_{n-kt,t})=kt+\reg  \left( \frac{R}{I_{n-kt,t}}\right).
		\]
  }
\medskip

 \noindent    where $\nu_1 (\Gamma_{n-kt,t})$ denotes the induced matching number of $\Gamma_{n-kt,t}$. 
 
We conclude this article with some open questions in Section~\ref{Section: Towards future works} and a conjecture on the bounds of the regularity of squarefree powers of $I(\Delta)$, where $\Delta$ is a simplicial tree.\\

{\bf Acknowledgements:} We sincerely thank the referee for their valuable suggestions, which have helped improve the presentation of this paper. We also thank Kanoy Kumar Das, Amit Roy, and Kamalesh Saha for carefully reading the draft, especially proof of Theorem 4.7, and providing helpful feedback.

\section{Squarefree powers and simplicial trees}\label{Section one: Preliminaries}

We first recall some basic concepts related to squarefree monomial ideals and simplicial complexes. The notation and definitions given in this section will be used throughout the later sections.

A \textit{simplicial complex} $\Delta$ on vertex set $V(\Delta)$ is a non-empty collection of subsets of $V(\Delta)$ such that  if $F'\in \Delta$ and $F \subseteq F'$, then $F \in \Delta$. Given a collection $\mathrm{F}=\{F_1,\dots,F_m\}$ of subsets of $V(\Delta)$, we denote by $\langle F_1,\dots,F_m\rangle$ or briefly $\langle \mathrm{F}\rangle$, the simplicial complex consisting of all subsets of $V(\Delta)$ which are contained in $F_i$ for some $i=1,\dots,m$. The elements of $\Delta$ are called \textit{faces} of $\Delta$. For any $F \in \Delta$, the dimension of $F$, denoted by $\dim(F)$  is one less than the cardinality of $F$. An \textit{edge} of $\Delta$ is a face of dimension $1$, while a \textit{vertex} of $\Delta$ is a face of dimension $0$. The dimension of $\Delta$ is given by $\max\{\dim(F):F\in \Delta\}$. The maximal faces of $\Delta$ with respect to the set inclusion are called \textit{facets}. We denote the set of all facets of $\Delta$ by $\F(\Delta)$. A {\em subcollection $\Delta'$} of $\Delta$ is a simplicial complex such that $\F(\Delta') \subseteq \F(\Delta)$. A subcollection $\Delta'$ of $\Delta$ is said to be {\em induced} if each facet $F \in \F(\Delta)$ with $F \subseteq V(\Delta')$ belongs to  $\Delta'$. A simplicial complex $\Delta$ is called \textit{pure} if all facets of $\Delta$ have the same dimension.	For a pure simplicial complex $\Delta$, the dimension of $\Delta$ is given trivially by the dimension of a facet of $\Delta$. 

Let $S=K[x_1,\dots,x_n]$ be a polynomial ring in $n$ variables over a field $K$. Let $I$ be a monomial ideal of $S$. We denote the unique set of minimal monomial generators of $I$ with $G(I)$. Given any $\{i_1,\dots,i_r\}\subset [n]=\{1, \ldots, n\}$,  we associate a squarefree monomial $u=x_{i_1}\dots x_{i_r} \in S$, and the set $\{i_1,\dots,i_r\}$ is called {\em support} of $u$, denoted by $\supp(u)$. Let  $\Delta$ be a simplicial complex on vertex set $[n]$. The monomial ideal generated by all squarefree monomials $x_{i_1}\dots x_{i_r}$ such that $ \{i_1,\dots,i_r\} \in \F(\Delta)$ is called the {\em facet ideal} of $\Delta$ and it is denoted by $I(\Delta)$. 

We highlight the following definitions and notation that describe different matching in $\Delta$. 

 \begin{enumerate}
     \item A {\em matching} of  $\Delta$ is a set of pairwise disjoint facets of $\Delta$. A matching consisting of $k$ facets is referred to as a \textit{$k$-matching}. 
     A $k$-matching  is called {\em maximal}, if $\Delta$ does not admit any $(k+1)$-matching. The {\em matching number} of $\Delta$ is the size of a maximal matching of $\Delta$ and is denoted by $\nu(\Delta)$.
     \item A matching $M$ of $\Delta$ is called {\em induced matching} if the set of facets of the induced subcollection on $\cup_{E \in M} E$ is $\langle M \rangle$. The {\em induced matching number }of $\Delta$ is the maximum size of an induced matching of $\Delta$ and denoted by $\nu_1(\Delta)$.
     \item In \cite{BHZ}, authors introduced the definition of a restricted matching for $1$-dimensional simplicial complexes (or simply a graph). We extend this definition to simplicial complexes of any given dimension in the following way. Let $F, G \in \mathcal{F}(\Delta)$. Then $F$ and $G$ form a {\em gap} in $\Delta$ if $F \cap G= \emptyset$ and the induced subcollection on vertex set $F\cup G  $ is $\langle F,G\rangle$. A matching $M$ of $\Delta$ is called a {\em restricted matching} if there exists a facet in $M$ forming  a gap with every other facet in $M$. The maximal size of a restricted matching of $\Delta$ is denoted by $\nu_0(\Delta)$ and it is called \textit{restricted matching number} of $\Delta$.
 \end{enumerate}
 
It immediately follows from the above definitions that $\nu_1(\Delta)\leq \nu_0(\Delta) \leq \nu(\Delta)$. We illustrate above definitions with following example. 

\begin{Example}\em \label{ex:restricted matching}
\begin{enumerate}
\item Consider the simplicial complex $\Delta$, whose $\mathcal{F}(\Delta)$ consists of facets $F_i=\{i, i+1, i+2\}$ for all $1\leq i\leq 10$. In Figure \ref{fig:simplicial complex}, we illustrate the geometric realization of $\Delta$. Then, $M=\{F_1, F_4, F_7, F_{10}\}$ is the maximal matching of $\Delta$, and hence $\nu(\Delta)= 4$. However, $M$ is not a restricted matching of $\Delta$. To see this, note that $F_1$ and $F_4$ do not form a gap because $F_2 \in \langle F_1 ,F_4\rangle$. Similarly, $F_7$ and $F_{10}$ do not form a gap because $F_8 \in \langle F_7 ,F_{10}\rangle$. Therefore, there does not exist a facet in $M$ that forms a gap with the other facets of $M$. 

However, $M'=\{F_1, F_5, F_8\}$ is a restricted matching of $\Delta$ because $F_1$ forms a gap with $F_5$ and $F_8$. This gives $\nu_0(\Delta)=3$. Furthermore, we observe that $M'$ is not an induced matching of $\Delta$ because $F_5$ and $F_8$ do not form a gap. Let $M''=\{F_1, F_5, F_9\}$; this is an induced matching of $\Delta$, and therefore $\nu_1(\Delta)=3$.

 \begin{figure}[h]
    \centering
\includegraphics[scale=0.7]{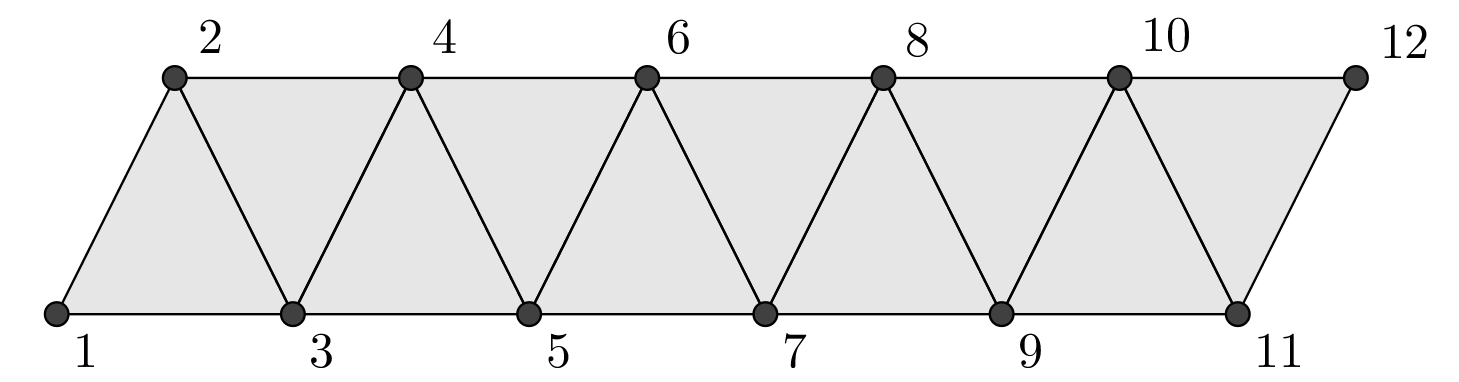}
  \caption{Geometric realization of $\Delta$.}
  \label{fig:simplicial complex}
\end{figure}

    \item Consider the simplicial complex $\Delta$ whose facets are all subsets of $[n]$ of size $k$ for some $1\leq k\leq n$, that is, $\mathcal{F}(\Delta)= \{ \{i_1, i_2, \ldots, i_k\}: 1\leq i_1< \ldots < i_k\leq n \}$. Then, $\nu_0(\Delta)=\nu_1(\Delta)=1$ and $\nu(\Delta)=\lfloor \frac{n}{k} \rfloor$. This example shows that, given a simplicial complex, the difference between its matching number and restricted matching number can be arbitrarily large.     
\end{enumerate}
\end{Example}

In \cite{F1}, Faridi introduced the class of the simplicial forests, which will be deeply discussed along this work. A facet $F$ of $\Delta$ is called a \textit{leaf} if either $F$ is the only facet of $\Delta$ or there exists a facet $G \in\Delta$ such that $H \cap F \subset G \cap F$ for all facets $H\neq F$ of $\Delta$. Such a facet $G$ is called {\em joint} of $F$ in $\Delta$. Observe that every leaf $F$ of $\Delta$ contains a \textit{free vertex}, that is a vertex $v$ of $\Delta$ such that $v \notin F'$ for every facet $F' \neq F$ of $\Delta$. A connected simplicial complex $\Delta$ is a \textit{tree} if every nonempty subcollection of $\Delta$ has a leaf. A simplicial forest is a simplicial complex whose every connected component is a tree. It directly follows from the definition of a simplicial tree that if $\Delta$ is a simplicial tree, then any subcollection of $\Delta$ is also a tree. For example, the simplical complex whose geometric realization is given in Figure~\ref{fig:simplicial complex}, is a simplicial tree. From \cite[Theorem 3.2]{HHTZ} we know that simple totally balanced hypergraphs are same as simplicial forests. A well known example of a simplicial forest is due to \cite[Corollary 2.9]{JV}, which we describe below.

A {\em rooted tree} $\Gamma$ is a directed tree with a fixed vertex, called the \textit{root}, which is directed implicitly away from the root. A directed path of length $t-1$ is a sequence of $t$ distinct vertices $i_1, \ldots ,i_t$, such that each pair $(i_j,i_{j+1})$ forms a directed edge from $i_j$ to $i_{j+1}$ for all $j=1, \ldots, t-1$.
 We denote by $\Gamma_t$ the simplicial complex whose facets correspond to the directed $t$-paths of $\Gamma$. In \cite[Theorem 2.7]{JV} it is shown that $\Gamma_t$ is a simplicial tree. Due to this, we refer to $\Gamma_t$ as {\em $t$-path simplicial tree of $\Gamma$}. The facet ideal of $\Gamma_t$ is called the \emph{$t$-path ideal} of $\Gamma$, and is given by 
\[
I (\Gamma_t)=(x_{i_1}\cdots x_{i_t}: i_1,\ldots ,i_t\text{ is a directed path on $t$ vertices in $\Gamma$}).
\]
 Now we recall briefly some definitions in graph theory, which will be useful in the next lemmas. Let $G$ be a simple graph with vertex set $V(G)$ and edge set $E(G)$. Two vertices $u,v \in V(G)$ are \textit{adjacent} if $\{u,v\} \in E(G)$. The \textit{degree} of a vertex $v$ of $G$, denoted by $\deg(v)$, is the number of the vertices adjacent to $v$ in $G$. A subgraph $H \subset G$ is a graph with $V(H)\subset V(G)$ and $E(H)\subset E(G)$. Moreover, $H$ is said to be an \textit{induced subgraph} on $V(H)$ if for any $u,v \in V(H)$, we have $\{u,v\} \in E(H)$ if and only if $\{u,v\} \in E(G)$.

A \textit{path} of length $r-1$ in $G$ is a sequence of $r$ distinct vertices $v_1,\dots,v_r$ such that $\{v_i,v_{i+1}\}$ is an edge of $G$, for all $i=1,\dots,r-1$. A \textit{cycle} of length $r$ is a sequence $v_{1},\dots,v_{r}, v_1$ such that $v_1, \ldots, v_r$ are distinct vertices of $G$, and $\{v_1,v_{r}\}$ and $\{v_{i},v_{i+1}\}$ are edges of $G$, for all $i=1,\dots,r-1$. A graph on $n$ vertices, forming a path of length $n-1$ is called a {\em path graph}, and is denoted by $P_n$. Let $v$ be a vertex of $P_n$ with degree one. Then $P_n$ can be viewed as a rooted tree with root $v$. The simplicial complex given in Example \ref{ex:restricted matching} (1) is a $3$-path simplicial tree of $P_{12}$ with edges $\{i, i+1\}$ for $i=1, \ldots, 11$.

Our aim is to study the squarefree powers of facet ideals of simplicial forests. Given a squarefree monomial ideal $I$ in $K[x_1, \ldots, x_n]$, the $k$-th squarefree power of $I$ is defined to be the ideal generated by squarefree elements of $G(I^k)$ and it is denoted by $I^{[k]}$. When $I$ is regarded as the facet ideal of a simplicial complex $\Delta$, the monomial generators of $I^{[k]}$ correspond to the $k$-matching of $\Delta$, that is, 
\[
I^{[k]}= (x_{i_1}\cdots x_{i_k} : \{i_1, \ldots, i_k\} \text{ is a $k$-matching of }\Delta).
\]
It is straightforward to check that $I^{[k]}\neq 0$ if and only if $1\leq k\leq \nu(\Delta)$. In particular, $I^{[\nu(\Delta)]}$ is the highest non-vanishing squarefree power of $I$.

Now, we give some preliminary results related to the matching of simplicial forests that will be used in subsequent sections.

\begin{Proposition}\label{M=N}
	Let $\Delta$ be a simplicial forest and $M=\{F_1, \ldots, F_r\}$ and $N=\{G_1, \ldots, G_r\}$ be two $r$-matchings of $\Delta$ with $\cup_{i=1}^rF_i= \cup_{i=1}^rG_i$. Then $M=N$. 
\end{Proposition}

\begin{proof}
	Assume that $M\neq N$.  Without loss of generality, we may assume that $M \cap N = \emptyset$. Consider the bipartite graph $G$ on the vertex set $\{F_1, \ldots, F_r\} \cup \{G_1, \ldots, G_r\}$ such that $\{F_i, G_j\} \in E(G)$ if and only if $F_i \cap G_j \neq \emptyset$. The condition  $\cup_{i=1}^rF_i= \cup_{i=1}^rG_i$ enforces that the degree of each vertex in $G$ is at least two. For instance, let $\{F_i, G_j\} \in E(G)$. Since $F_i \neq G_j$ and $\cup_{i=1}^tF_i= \cup_{i=1}^tG_i$, there exists some $F_k \in M$ with $i\neq k$ such that $(G_j \setminus F_i )\cap F_k \neq \emptyset$. Hence $\deg(G_j)\geq 2$.
 
	The fact that every vertex of $G$ has degree at least two shows that $G$ is not a tree and therefore contains a cycle $C$ of length $t\geq 4$.  After rearranging the indices, let $C$ be given by: $F_1, G_1, F_2, \ldots, F_t, G_t, F_1$. Consider the subcollection  $\Delta'=\langle F_1, \ldots, F_t, G_1, \ldots, G_t\rangle$ of $\Delta$. Then $\Delta'$ is also a simplicial forest with a leaf, say $G_1$. Since $M$ is a matching, we observe that $G_1 \cap F_i$ is disjoint with $G_1 \cap F_j$ for all $i \neq j$. Moreover, the sets  $G_1 \cap F_1$ and $G_1 \cap F_2$ are non-empty by the definition of $G$. This shows that $G_1$ is not a leaf of $\Delta'$, and $\Delta'$ is a not a simplicial tree, which is a contradiction to $\Delta$ being a simplicial tree. 
\end{proof}

We obtain following description of elements of $ G(I(\Delta)^{[k]})$ as an immediate corollary of above proposition. 

\begin{Corollary}\label{cor:M=N1}
	Let $\Delta$ be a simplicial forest with  $\F(\Delta)=\{F_1, \ldots, F_r\}$ and for all $i=1, \ldots, r$, set $f_i=\prod_{j \in F_i}x_j$. For any $1 \leq k \leq \nu(\Delta)$, each $ u \in G(I(\Delta)^{[k]})$ can be uniquely expressed as $u=f_{i_1}\cdots f_{i_k}$ where $M=\{F_{i_1}, \ldots, F_{i_k}\}$ is a $k$-matching of $\Delta$. 
\end{Corollary}

As noted in Example~\ref{ex:restricted matching}(2) that for an arbitrary simplicial complex, the difference between its matching number and restricted matching number can be arbitrarily large. It is argued in \cite[page 24]{BHZ} that if a simple graph $G$ is a tree then $\nu(G) - \nu_0(G) \leq 1$. Below, we extend this result to any simplicial forest. 

\begin{Proposition}\label{prop:difference of matchings}
	Let $\Delta$ be a simplicial forest. Then $\nu(\Delta) - \nu_0(\Delta) \leq 1$. 
\end{Proposition}

\begin{proof}
	Let $s=\nu(\Delta)$ and $M=\{E_1, \ldots, E_s\}$ be a $s$-matching  of $\Delta$. Consider the graph $G_M$ associated with $M$ such that $V(G_M)= M$ and \[
	E(G_M)=\{\{E_i,E_j\}: \; E_i \text{ and } E_j \text{ do not form a gap in } \Delta \}.
	\] 
	
	We first prove that $G_M$ is a forest. To do this, suppose that $G_M$ is not a forest; that is, $G_M$ contains a cycle of length at least three. We may assume that the sequence of vertices $E_1, \ldots, E_r$ gives a minimal cycles in $G_M$ for some $r \leq s$. By the definition of $G_M$, for each $i=1, \ldots, r$, there exists $F_i \in \F(\Delta)$ such that $F_i \subset  E_i \cup E_{i+1}$ and $F_i \not\in  \langle E_i , E_{i+1}\rangle$ .
	
	Let  $\Delta'$ be the subcollection of $\Delta$ such that $F_1,  \ldots,F_r \in \Delta'$ and if $F_{i-1} \cap F_i = \emptyset$ for some $i\in\{2, \ldots, r\}$ then $E_i \in \Delta'$. Since $\Delta$ is a simplicial forest, the subcollection $\Delta'$ is also a simplicial forest and contains a leaf. For all $i=1, \ldots, r$, using the fact that $M$ is a matching of $\Delta$ provides that  if $F_i \cap F_j \neq \emptyset$, then $j \in \{i-1, i+1\}$, and if $F_i \cap E_j \neq \emptyset$, then $j \in \{i, i+1\}$. Therefore, if $F_i$ is a leaf of $\Delta'$, then the only possible candidates to be a joint of $F_i$ in $\Delta'$ are $F_{i-1}, F_{i+1}$ and $E_i, E_{i+1}$ provided that $E_i, E_{i+1} \in \Delta'$. We have the following possible cases:
	
	\begin{itemize}
		\item[(1)] Let $E_i, E_{i+1} \in \Delta'$. Then $F_i \cap F_{i+1}= \emptyset$ and $F_{i-1} \cap F_i= \emptyset$. Moreover,  $F_i\cap E_i$ and $F_i \cap E_{i+1}$ are nonempty and disjoint. This shows that $F_i$ is not a leaf of $\Delta'$.
  
		\item[(2)] Let $E_i \in \Delta'$ and $E_{i+1} \notin  \Delta'$. Then $F_i \cap F_{i+1}\neq \emptyset$ and $F_{i-1} \cap F_i = \emptyset$ . Since $F_i \cap F_{i+1}\subset E_{i+1}$, we see that $F_i \cap F_{i+1}$ and $E_i \cap F_i$ are disjoint. Hence $F_i$ is not a leaf of $\Delta'$. The case when $E_i \notin \Delta'$ and $E_{i+1} \in  \Delta'$ can be argued in a similar way to conclude that $F_i$ is not a leaf of $\Delta'$. 
  
		\item[(3)]	Let $E_i, E_{i+1} \notin \Delta'$. Then $F_i\cap F_{i-1} \neq \emptyset$ and $F_i\cap F_{i-1} \neq \emptyset$. It follows from the fact that $E_i \cap E_{i+1} = \emptyset$ and  $F_i\cap F_{i-1} \subset E_i$ and  $F_i\cap F_{i+1} \subset E_{i+1}$ that $F_i$ is not a leaf of $\Delta'$. 		
	\end{itemize}
	The above discussion shows that $F_i$ is not a leaf of $\Delta'$ for any $i=1, \ldots, r$. 	If $E_i \notin \Delta'$ for all $i=1, \ldots, r$, then $\Delta'$ does not contain a leaf, which is a contradiction to $\Delta$ being a simplicial forest. Let $E_i \in \Delta'$ for some $i= 1, \ldots, r$ such that $E_i$ is a leaf of $\Delta'$. Then $F_{i-1} \cap F_i = \emptyset$. The only possible joints of $E_i$ in $\Delta'$ are $F_{i-1}$ and $F_i$ because $E_i$ does not intersect any other facet non-trivially. On the other hand,we observe that $E_i \cap F_{i-1}$ and $E_i \cap F_i$ do not contain each other because $F_{i-1} \cap F_i = \emptyset$. From this we conclude that $E_i$ is not a leaf of $\Delta'$, a contradiction to the assumption that $\Delta$ is a simplicial forest. Therefore, $G_M$ does not contain any cycle and the claim holds. 
	
	If $G_M$ contains an isolated vertex, say $E_i$, then $E_i$ forms a gap with all other elements of $M$. This gives $\nu_0(\Delta)=\nu(\Delta)$. If $G_M$ does not contain any isolated vertex, then pick $E_i$ such that $E_i$ is a leaf of $G_M$ and let $E_j$ be the unique neighbor of $E_i$ in $G_M$. Then $M \setminus \{E_j\}$ forms a restricted matching of $\Delta$ and $\nu_0(\Delta)=\nu(\Delta)-1$.	
\end{proof}

\section{simplicial trees with linear squarefree powers}\label{Section: Simplicial forests with linear squarefree powers}

Let $S=K[x_1,\dots,x_n]$ and $I$ be a homogeneous ideal of $S$. From the famous Hilbert’s Syzygy Theorem it is well-known that a minimal graded free resolution $\mathbb{F}(I)$ of $I$ exists, it is unique up to isomorphisms and it is finite with length at most $n$. In such a case, $\mathbb{F}(I)$ can be written as
\begin{small}
	\[ 
	0\rightarrow \bigoplus_{j\in \ZZ} S(-j)^{\beta_{\ell,j}}  \xrightarrow{d_\ell}\cdots \rightarrow  \bigoplus_{j\in \ZZ} S(-j)^{\beta_{i,j}} \xrightarrow{d_i}   \cdots \rightarrow   \bigoplus_{j\in \ZZ} S(-j)^{\beta_{0,j}} \xrightarrow{d_0} I \rightarrow 0,
	\]
\end{small}
where $\ell \leq n$. The numbers $\beta_{i,j}$ are called the \textit{graded Betti numbers} of $I$ and the \textit{Castelnuovo-Mumford regularity} (or simply \textit{regularity}) of $I$ is $\reg(I)=\max\{j:\beta_{i,i+j}\neq 0,\text{ for some } i \}$. Moreover, $\reg(I)=\reg(S/I)+1.$  If $\beta_{i,j}(I) = 0$ for all $i \geq 0$ and for $j \neq i + t$, then $I$ admits a \textit{linear resolution}. From \cite[Proposition 8.2.1]{HH1} we know that $I$ has a linear resolution if $I$ has linear quotients, that is if there exists a system of homogeneous generators $f_1, f_2, \dots, f_m$ of $I$ such that $(f_1,\dots, f_{i-1}): f_i$ is generated by linear forms, for all $i=1,\dots,m$. 

 \subsection{Simplicial trees with linear squarefree powers} 
 In this subsection, our aim is to show that if the facet ideal of a simplicial tree has linear resolution then its non-vanishing squarefree powers also have linear resolution. Zheng in \cite{Z} characterized all simplicial trees whose facets ideals have linear resolution. Indeed, it is shown in \cite[Theorem 3.17]{Z} that the facet ideal of a simplicial tree $\Delta$ has linear resolution if and only if $\Delta$ satisfies the intersection property. However, it turns out that for any simplicial tree $\Delta$ with the intersection property, we have $\nu(\Delta)\leq 2$, as shown in Proposition \ref{lem:matchingnumber}. So, if $\Delta$ has the intersection property then $I(\Delta)$ has at most two non-vanishing squarefree powers. To prove Proposition \ref{lem:matchingnumber}, we first recall some definitions from \cite{Z}. 


 A simplicial complex $\Delta$ is said to be {\em connected in codimension 1}, if for any two facets $F$ and $G$ of $\Delta$ with $\dim(F) \geq \dim(G)$, there exists a chain $\mathcal{C} : F= F_0, \ldots, F_n=G$ between $F$ and $G$ such that $\dim(F_i \cap F_{i+1} )= \dim(F_{i+1})-1$ for all $i=0, \ldots, n-1$. Such a chain $\mathcal{C}$ is called a {\em proper} chain. A proper chain $\mathcal{C} $ between $F$ and $G$ is called {\em irredundant} if no subsequence of this chain except $\mathcal{C} $ itself is a proper chain between $F$ and $G$.
 
 Let $\Delta$ be a pure $d$-dimensional simplicial tree connected in codimension 1. It is known from \cite[Proposition 1.17]{Z} that for any two facets $F$ and $G$, there exists a unique irredundant proper chain between $F$ and $G$. The length of the unique irredundant proper chain between $F$ and $G$ is called the {\em distance} between $F$ and $G$, and is denoted by $\mathrm{dist}(F,G)$. If for any two facets $F$ and $G$ with $\dim (F \cap G)= d-k$ for some $k\in \{1, \ldots, d+1\}$, we have $\mathrm{dist} (F,G)= k$, then $\Delta$ is said to have the {\em intersection property}. 
 
 \begin{Proposition} \label{lem:matchingnumber}
 	Let $\Delta$ be a simplicial  tree with the intersection property and $\dim(\Delta) \geq 1$. Then $\nu(\Delta)\leq 2$. 
 \end{Proposition}
 \begin{proof}
 	From the definition of the intersection property, we know that $\Delta$ is pure. Let $\dim \Delta = n-1$. 
  
  On contrary, assume that $\nu(\Delta)\geq 3$, and let $E_1,E_2$ and $E_3$ be three pairwise disjoint facets of $\Delta$.  Let $E_1=\{a_1,\ldots,a_n\}$, $E_2=\{b_1,\ldots ,b_n\}$ and $E_3=\{c_1,\ldots,c_n\}$. Since $\Delta$ has intersection property, using $\dim(E_i\cap E_j)=-1$, we obtain $\mathrm{dist}(E_i,E_j)=n$, for all $i\neq j$.   Let $ \mathcal{C}_1:\ \ E_1=F_0,F_1, \ldots ,F_n=E_2$ be the unique  irredundant chain between $E_1$ and $E_2$. Since  $\dim(F_i\cap F_{i+1})=\dim(F_{i+1})-1$, after rearranging indices, we set $F_i=\{b_1, \ldots, b_i, a_{i+1}, \ldots, a_n\}$, for all $i=1, \ldots, n-1$. Similarly, let $ \mathcal{C}_2:\ \ E_2=F'_0,F'_1, \ldots ,F'_n=E_3$ be the  unique  irredundant chain between $E_2$ and $E_3$ with $F'_i=\{c_1, \ldots, c_i, b_{i+1}, \ldots, b_n\}$ for all $i=1, \ldots, n-1$. Lastly, let $  \mathcal{C}_3:\ \ E_3=F''_0,F''_1, \ldots ,F''_n=E_1 $ be the  unique  irredundant chain between $E_1$ and $E_3$ with $F''_i=\{a_1, \ldots, a_i, c_{i+1}, \ldots, c_n\}$ for all $i=1, \ldots, n-1$.
 	
 	Consider $\Delta'=\langle E_1, E_2, E_3, F_{n-1}, F'_{n-1}, F''_{n-1} \rangle \subseteq \Delta$. Observe that $E_1\cap  F_{n-1}=\{a_n\}$, $E_1 \cap  F''_{n-1}=\{a_1, \ldots, a_{n-1}\}$ and the intersection of $E_1$ with $E_2, E_3$ and $F'_{n-1}$ is trivial. This shows that $E_1$ is not a leaf of $\Delta'$. Similarly, one can see that none of the facet of $\Delta'$ is a leaf of $\Delta'$, and hence $\Delta'$ is not a simplicial tree, which is a contradiction to $\Delta$ being a simplicial tree.  Therefore, we conclude  $\nu(\Delta)\leq 2$.
 \end{proof}
We conclude this subsection with the following result which can be seen as a consequence of \cite[Lemma 2.3]{KK1} and \cite[Proposition 2.10]{JZ}.

\begin{Theorem}\label{thm: simplicial tree with i.p. has linear quotients}
    Let $\Delta$ be a simplicial tree with intersection property. Then, $I(\Delta)$ and ${I(\Delta)}^{[2]}$ has linear quotients.
\end{Theorem}
\begin{proof}
  Let $\Delta$ be a simplicial tree with the intersection property. Zheng's result \cite[Proposition 3.9 and Theorem 3.17]{Z} gives that $I(\Delta)$ has linear quotients, and this result is generalized in \cite[Lemma 2.3]{KK1} which proves that ${I(\Delta)}^{k}$ has linear quotients for all $k$. Moreover, \cite[Proposition 2.10]{JZ} asserts that if a monomial ideal $J$ has linear quotients, then the ideal generated by the squarefree monomials in $J$ also has linear quotients. Therefore, the desired conclusion follows from \cite[Proposition 2.10]{JZ} together with \cite[Lemma 2.3]{KK1}. 
\end{proof}

\subsection{The $\nu$-th  squarefree power of simplicial trees}
Let $\Delta$ be a simplicial tree. In this subsection, we study the linearity of the highest non-vanishing squarefree power, namely the $\nu$-th squarefree power of $I(\Delta)$.  It is shown in \cite[Theorem 5.1]{BHZ}, that the $\nu$-th squarefree power of the edge ideal of a simple graph has linear quotients. Such a statement does not hold for arbitrary squarefree monomial ideals, or even for the facet ideals of simplicial trees, as shown in the following example.

\begin{Example}\em\label{exa: 3-path of a rooted tree with no linear resolution}
Let $\Gamma$ be the rooted tree in Figure \ref{fig:rooted tree} and $I(\Gamma_3)$ be the $3$-path ideal of $\Gamma$. Then $I(\Gamma_3)=(x_1x_2x_4,x_1x_2x_5,x_1x_3x_6, x_1x_3x_7,x_2x_4x_8,x_2x_5x_9, x_3x_6x_{10}, x_3x_7x_{11}).$
\begin{figure}[h]
    \centering
    \includegraphics[scale=0.9]{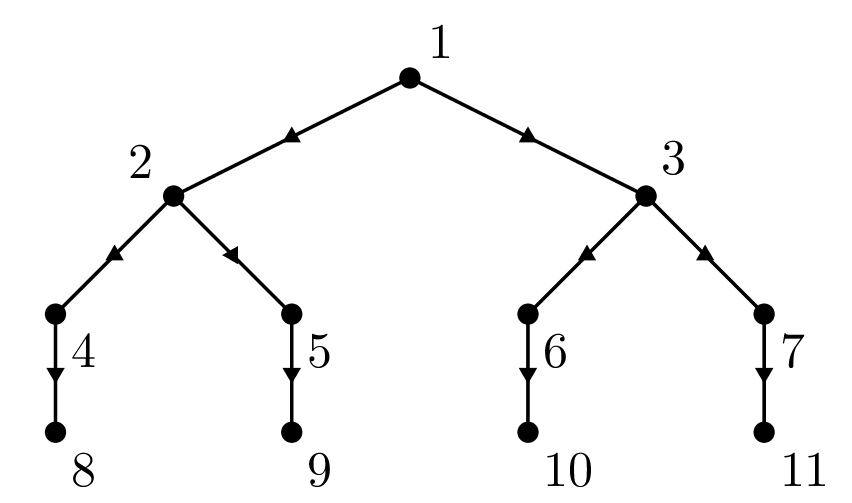}
    \caption{A rooted tree.}
    \label{fig:rooted tree}
\end{figure}
Observe that $\nu(\Gamma_3)=2$, and $I(\Gamma_3)^{[k]}=0$ for all $k\geq3$. Easy computations with \texttt{Macaulay2} \cite{M2} show that $I(\Gamma_3)^{[2]}$ does not have a linear resolution. This shows that \cite[Theorem 5.1]{BHZ} cannot be extended even to the case of $t$-path ideals of rooted trees. 
\end{Example}

In view of Example 2.3, we are led to ask for which classes of simplicial trees one can generalize \cite[Theorem 5.1]{BHZ}. More precisely, we want to address the following question.
\begin{Question}
Can we find some classes of simplicial trees whose facet ideals have the property that their $\nu$-th squarefree powers have linear resolutions. 
\end{Question}
 In an effort to answer above question, we prove that an analogue of \cite[Theorem 5.1]{BHZ} holds true for $I(\Gamma_t)$, where $\Gamma_t$ is a simplicial tree whose facets are the directed $t$-paths of a special rooted tree $\Gamma$ known as broom graph, see \cite{BHaK}. 
 A \emph{broom graph} $\Gamma$ of height $h$ is a rooted tree with root $x$, consisting of a {\em handle} which forms a directed path of length $h$ rooted at $x$, and all the other vertices that are not on the handle are leaves of $\Gamma$. See Figure~\ref{Broom Graph} for an example of a broom graph. A graph consisting of only a directed path is a broom graph consisting of only handle. Before stating our next theorem related to broom graphs, we setup the following notation. 

\begin{Notation}\label{rem:facetofgamma}
{\em Let $\Gamma$ be a broom graph rooted at the vertex $x_{0,0}$ with $\mathrm{ht}(\Gamma)=h$ and let $x_{0,0},x_{1,0}\ldots ,x_{h,0}$ be the vertices of the handle of $\Gamma$. Furthermore, for each $1 \leq i \leq h$, let $l_i$ be the number of vertices of $\Gamma$ which do not lie on the handle and their unique neighbor on the handle is $x_{i-1,0}$. We set the following notation for the facets of $\Gamma_t$ where $ t\geq 2$.  For $0\leq i\leq h-t+1$ and $0\leq j\leq \ell_{i+t-1}$, set 
		\begin{equation}\label{Eq 1}
			F_{i,j}=\{x_{i,0},x_{i+1,0}, \ldots ,x_{i+t-2,0},x_{i+t-1,j}\}.
		\end{equation}
	We define a total order on $\F(\Gamma_t)$ as follows: for all $F_{i,j}, F_{k,m} \in \F(\Gamma_t)$ with $i,k \in \{0, \ldots, h-t+1 \}$, $0\leq j\leq \ell_{i+t-1}$ and $0 \leq m \leq \ell_{k+t-1}$, we set $F_{i,j} < F_{k,m} $ if either $i < k$ or $i=k$ and $m< j$. 
		
We identify the vertices of $\Gamma$ as variables and set $S=K[x_{i,j} : x_{i,j} \in V(\Gamma)]$. For each $F_{i,j} \in \F(\Gamma_t)$, let $m_{i,j}$ be the monomial corresponding to $F_{i,j}$, that is $m_{i,j}=x_{i,0}x_{i+1,0}\cdots x_{i+t-2,0}x_{i+t-1,j}$. Let $ 1 \leq k \leq \nu(\Gamma_t)$.  Due to Corollary~\ref{cor:M=N1},  any monomial generator of $I^{[k]}$ can be uniquely expressed as $m_{i_1, j_1}\cdots m_{i_k, j_k}$ such that $M=\{F_{i_1, j_1}, \ldots, F_{i_k, j_k}\}$ is a $k$-matching. Let $u_a, u_b \in G(I^{[k]})$ with $u_a=m_{i_1,j_1}\cdots m_{i_k,j_k}$ and $u_b=m_{i'_1,j'_1}\cdots m_{i'_k,j'_k}$, such that $F_{i_1,j_1}< \ldots < F_{i_k,j_k}$ and $F_{i'_1,j'_1} < \ldots < F_{i'_k,j'_k}$ and $s=\max\{\ell: F_{i_\ell,j_\ell}\neq F_{i'_\ell,j'_\ell}\}$. We set $u_b <u_a$ if $F_{i'_s,j'_s}< F_{i_s,j_s}$. 
}
\end{Notation}

\begin{Example}\em
	\begin{figure}[h]
		\centering
		\includegraphics[scale=0.8]{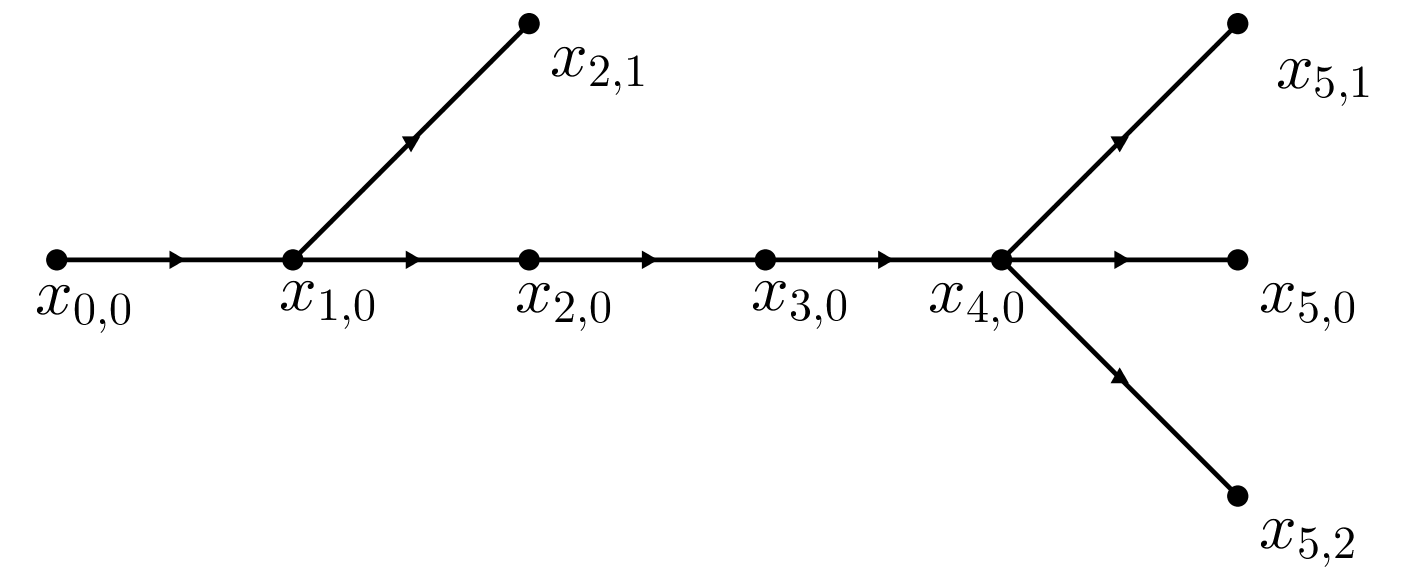}
		\caption{A broom graph.}
		\label{Broom Graph}
	\end{figure}
	Let $\Gamma$ be the broom graph as in Figure \ref{Broom Graph} and $\Gamma_3$ be its $3$-path simplicial tree. Then $\mathrm{ht}(\Gamma)=5$ and  $l_1=l_3=l_4=0$, $l_2=1$ and $l_5=2$. Following Notation~\ref{rem:facetofgamma}, we label the facets of $\Gamma_3$ as $ F_{0,0}=\{x_{0,0},x_{1,0},x_{2,0}\}$, $F_{0,1}=\{x_{0,0},x_{1,0},x_{2,1}\}$,$F_{1,0}=\{x_{1,0},x_{2,0},x_{3,0}\}$, $F_{2,0}=\{x_{2,0},$ $x_{3,0},x_{4,0}\}$, $F_{3,0}=\{x_{3,0},x_{4,0},x_{5,0}\}$, $F_{3,1}=\{x_{3,0},x_{4,0},x_{5,1}\}$, and $F_{3,2}=\{x_{3,0},$ $x_{4,0},x_{5,2}\}$.  With the total order defined on the facets of $\Gamma_t$ in Notation~\ref{rem:facetofgamma} we obtain $F_{3,0}>F_{3,1}>F_{3,2}>F_{2,0}>F_{1,0}>F_{0,0}>F_{0,1}$. 	Moreover, the elements of $G(I^{[2]} )$ are ordered as 	
	 \[
	 m_{0,0}m_{3,0} >  m_{0,1}m_{3,0} > m_{0,0} m_{3,1}> m_{0,1} m_{3,1}> m_{0,0}m_{3,2} >  m_{0,1} m_{3,2}  > m_{0,1} m_{2,0}.
	 \]
\end{Example}

Now, we give the main result of this subsection. As shown in Proposition~\ref{prop:difference of matchings}, given a simplicial tree $\Delta$, either $\nu(\Delta)= \nu_0(\Delta)$ or $\nu(\Delta)= \nu_0(\Delta)+1$. Thus, the $\nu_0$-th squarefree power of $I(\Delta)$ is either the highest non-vanishing power or the second highest non-vanishing power of  $I(\Delta)$. In Theorem~\ref{thm: k-squarefree power non-linearly related}, we show that $I(\Delta)^{[k]}$ does not have a linear resolution for all $1\leq k < \nu_0(\Delta)$. Therefore, it is interesting to study the
$\nu_0$-th and $\nu$-th squarefree power of $I(\Delta)$. 

\begin{Theorem}\label{thm:broomgraph-linearquotient}
Let $\Gamma$ be a broom graph and $\Gamma_t$ be its $t$-path simplicial tree. Moreover, let $I$ be the $t$-path ideal of $\Gamma$, that is, $I=I(\Gamma_t)$.
	\begin{enumerate}
		\item[{\em(i)}] $I^{[\nu(\Gamma_t)]}$ has linear quotient.
		\item[{\em (ii)}]  In addition, if $\Gamma$ is a path graph, then ${I}^{[\nu_0(\Gamma_t)]}$ has linear quotient. 
	\end{enumerate} 
\end{Theorem}
\begin{proof}	Set $I=I(\Gamma_t)$. We use Notation~\ref{rem:facetofgamma} to order the elements of $G(I^{[q]})$ as $u_1 > \ldots > u_p$, and use this ordering of $G(I^{[q]})$ to prove $(i)$ and $(ii)$. Let $u_a, u_b \in G(I^{[q]})$  with $a<b$ and $u_a=m_{i_1,j_1}\cdots m_{i_q,j_q}$, $u_b=m_{i'_1,j'_1}\cdots m_{i'_q,j'_q}$ such that $F_{i_1,j_1}< \ldots < F_{i_q,j_q}$ and $F_{i'_1,j'_1} < \ldots < F_{i'_q,j'_q}$. Moreover, let $s=\max \{\ell:F_{i_\ell,j_\ell}\neq F_{i'_\ell,j'_\ell}\}$. Since, $u_b< u_a$, we obtain $F_{i'_s,j'_s}<F_{i_s,j_s}$.   Let $M=\{F_{i_1,j_1}, \ldots , F_{i_q,j_q}\}$ and $N=\{F_{i'_1,j'_1} , \ldots , F_{i'_q,j'_q}\}$. To prove (i) and (ii), we invoke \cite[Lemma 8.2.3]{HH1} and construct a monomial $u_c \in G(I^{[q]})$ with $u_b<u_c$ such that  $(u_c)\colon (u_b)$ is generated by a variable and $(u_a)\colon (u_b)\subseteq (u_c)\colon(u_b)$. 
	
	$(i)$ Let $q=\nu(\Gamma_t)$. Since $ F_{i'_s,j'_s}<F_{i_s,j_s} $, it follows that either $i'_s<i_s$ or $i_s=i'_s$ and $j_s<j'_s$.  We distinguish three cases: (1) $i_s=i'_s$; (2) $i'_s<i_s$  and $0<j'_s$; (3) $i'_s<i_s$  and $j'_s=0$. 
	
	Case (1): Let $i_s=i'_s$. We know that $F_{i'_s,j_s}$ does not intersect any facet in $\{F_{i'_{s+1}, j'_{s+1}}, \ldots,$ $ F_{i'_q, j'_q}\}=\{F_{i_{s+1}, j_{s+1}}, \ldots, F_{i_q, j_q} \}\subset M$ because $M$ is a matching of $\Gamma_t$. The construction of facets of $\Gamma_t$ in (\ref{Eq 1}) together with the fact that $F_{i'_s,j'_s}$ does not intersect any facet in $\{F_{i'_1, j'_1}, \ldots, F_{i'_{s-1}, j'_{s-1}}\}$ shows that $F_{i'_s,j_s}$ also does not intersect any facet in $\{F_{i'_1, j'_1}, \ldots, F_{i'_{s-1}, j'_{s-1}}\}$. From this we conclude that $F_{i'_s,j_s}$ does not intersect any facet in $N\setminus \{F_{i'_s, j'_s}\}$. Hence $u_c=(u_b/ m_{i'_s,j'_s})m_{i'_s,j_s} =(u_b/x_{i'_s+t-1, j'_s}) x_{i'_s+t-1, j_s}	\in G(I^{[q]})$, and $(u_c)\colon (u_b)=(x_{i'_s+t-1, j_s})$. Since $F_{i'_s,j'_s}<F_{i'_s,j_s}$, it follows that $u_b<u_c$. 
	Since $x_{i'_s+t-1, j_s} \in F_{i'_s,j_s}\subseteq  \supp(u_a)$,  we have $  (u_a)\colon (u_b) \subseteq(x_{i'_s+t-1, j_s}) $, as required. 
	
Before proceeding further to Case(2) and Case(3), we need to acknowledge that  $F_{i_s,j_s}\cap F_{i'_s,j'_s}\neq \emptyset$. Indeed, since $F_{i'_s,j'_s}$ is disjoint with all the facets in $\{F_{i'_1,j'_1}, \cdots, F_{i'_{s-1},j'_{s-1}}\}$, and $F_{i'_1,j'_1} < \ldots < F_{i'_s,j'_s}<F_{i_s,j_s}$, it follows immediately that  $F_{i_s, j_s}$ is also disjoint with each of the facet in $\{F_{i'_1,j'_1}, \cdots, F_{i'_{s-1},j'_{s-1}}\}$. Moreover, due to $\{F_{i'_{s+1}, j'_{s+1}}, \ldots, F_{i'_q, j'_q} \}=\{F_{i_{s+1}, j_{s+1}}, \ldots, F_{i_q, j_q} \}\subset M$, the facet $F_{i_s,j_s} $ is disjoint with each of the facet in $\{ F_{i'_{s+1},j'_{s+1}}, \ldots, F_{i'_q,j'_q}\}$. Therefore, if $F_{i_s,j_s}\cap F_{i'_s,j'_s}= \emptyset$, then
	$F_{i_s,j_s} \cup N$ forms a $(q+1)$-matching, a contradiction to the assumption $q=\nu(\Gamma_t)$. Therefore, $F_{i_s,j_s}\cap F_{i'_s,j'_s}\neq \emptyset$. To discuss Case(2) and Case(3) we make use of $F_{i_s,j_s}\cap F_{i'_s,j'_s}\neq \emptyset$. Note that the vertices in $F_{i_s,j_s}\cap F_{i'_s,j'_s}$ lie on the handle of $\Gamma$.
	
	Case (2):  Let $i'_s<i_s$  and $0<j'_s$. Then $F_{i_s,j_s}\cap F_{i'_s,j'_s}\neq \emptyset$ if and only if $i'_s+1 \leq i_s \leq i'_s+t-2$. This shows that $x_{i'_s+t-1, 0}  \in F_{i_s,j_s} \subset \supp(u_a)$. Also $x_{i'_s+t-1, 0}  \not\in F_{i'_s,j'_s}$ due to $j'_s>0$ and therefore $x_{i'_s+t-1, 0}  \not\in \supp(u_b)$. This gives $ (u_a)\colon (u_b) \subseteq (x_{i'_s+t-1, 0})   $. Let $u_c=	(u_b/ m_{i'_s,j'_s})m_{i'_s+1,0}$, Then $(u_c)\colon (u_b)=(x_{i'_s+t-1, 0})$. Since $F_{i'_s,j'_s}<F_{i'_s+1,0}$, it yields $u_b<u_c$.  Again from $F_{i'_s,j'_s}<F_{i'_s+1,0}$,  it is follows immediately that $F_{i'_s+1,0}$ does not intersect any facet in $\{F_{i'_1,j'_1}, \ldots ,  F_{i'_{s-1},j'_{s-1}}\}$. Because of $i'_s+1 \leq i_s \leq i'_s+t-2$, we have either  $F_{i'_s+1,0} = F_{i_s,j_s} $ or $F_{i'_s+1,0} < F_{i_s,j_s}  $. In both case, it is easy to see that $F_{i'_s+1,0}$ does not intersect any element in $\{F_{i'_{s+1}, j'_{s+1}}, \ldots, F_{i'_q, j'_q} \}=\{F_{i_{s+1}, j_{s+1}}, \ldots, F_{i_q, j_q} \}\subset M$. Hence $u_c \in G(I^{[q]})$. Then $ (u_a)\colon (u_b)\subseteq (x_{i'_s+t, 0})  = (u_c)\colon (u_b) $, as required.

	Case (3):  Let $i'_s<i_s$  and $j'_s=0$.  Then $F_{i_s,j_s}\cap F_{i'_s,j'_s}\neq \emptyset$ if and only if $i'_s+1 \leq i_s \leq i'_s+t-1$. If $i'_s+1 =i_s $, then set $u_c=	(u_b/ m_{i'_s,j'_s})m_{i'_s+1,j_s}$, and if $i'_s+1 < i_s \leq i'_s+t-1$ then set $u_c=	(u_b/ m_{i'_s,j'_s})m_{i'_s+1,0}$. In both cases, $u_b<u_c$ and $ (u_a):(u_b)\subset (u_c):(u_b) $, and by arguing similarly as in the cases above, it follows that $u_c \in G(I^{[q]})$. Also, if $i'_s+1 =i_s $ then $(u_c)\colon (u_b)=(x_{i'_s+t, j_s})$, and if $i'_s+1 < i_s \leq i'_s+t-1$ then  $(u_c)\colon (u_b)=(x_{i'_s+t, 0})$. This completes the proof. 

	$(ii)$  Now, let $\Gamma$ be a path graph. It is known from Proposition~\ref{prop:difference of matchings} that $\nu(\Gamma_t)-\nu_0(\Gamma_t)\leq 1$. If $\nu(\Gamma_t)-\nu_0(\Gamma_t)=0$, then the assertion follows by virtue of $(i)$. Therefore, it is enough to consider the case when $\nu_0(\Gamma_t)=\nu(\Gamma_t)-1$. Set $q=\nu_0(\Gamma_t)$. 
	
	Since $\Gamma$ is a path graph, we can view  $\Gamma$ as a broom graph consisting of only the handle $x_{0,0}, x_{1,0}, \ldots, x_{h,0}$. In this case, $j=0$ in (\ref{Eq 1}). To simplify the notation, we set $F_i:=F_{i,0}$ and $m_i:=m_{i,0}$ for each $0\leq i \leq h-t+1$.
	
	To construct the monomial $u_c$, we proceed in the following way. Since $F_{i'_s}< F_{i_s}$, we obtain $i'_s<i_s$.  If $F_{i_s}\cap F_{i'_s}\neq \emptyset$, then the desired conclusion follows from the same argument as in Case (3) above. Now suppose that $F_{i_s}\cap F_{i'_s} = \emptyset$, that is , $i'_s+t-1 < i_s$. 
 Then due to $F_{i'_1}< \ldots < F_{i'_s}<F_{i_s}$, we conclude that  $F_{i_s}$ does not intersect any facet in $\{F_{i'_1}, \ldots,  F_{i'_s}\}$. Also, due to $\{F_{i'_{s+1}}, \ldots, F_{i'_q} \}=\{F_{i_{s+1}}, \ldots, F_{i_q} \}\subset M$, the facet $F_{i_s} $ is disjoint with each of the facet in $\{ F_{i'_{s+1}}, \ldots, F_{i'_q}\}$. Therefore, $A=\{F_{i'_1},\ldots,F_{i'_s},   F_{i_s} , F_{i'_{s+1}}\ldots,F_{i'_q}\}$ forms a maximal matching of $\Gamma_t$.
	
	First we claim that $F_{i'_s}$ and $   F_{i_s}$ do not form a gap in $\Gamma_t$. To prove the claim, assume that $F_{i'_s}$ and $   F_{i_s}$ form a gap in $\Gamma_t$. Then $   F_{i_s}$ forms a gap with all elements in $\{F_{i'_1},\ldots,F_{i'_s}\}$. It yields $s \neq q$, otherwise $A$ is a restricted matching of size $q+1$, a contradiction. Also $i'_s+t \neq  i_s$, otherwise $F_{i'_s+t-1}$ belongs to the induced subcollection generated by $F_{i'_s}$ and $   F_{i_s}$, a contradiction to our assumption that  $F_{i'_s}$ and $   F_{i_s}$  form a gap. Since $i'_s+t-1 < i_s$ and  $i'_s+t \neq  i_s$, we conclude $i'_s+t <  i_s$. Note that $F_{i'_s} \cap F_{i_s-1} =\emptyset$. Consider the following set 	
	\[
	B=\{F_{i'_1},\ldots,F_{i'_s},   F_{i_s-1} , F_{i'_{s+1}-1}\ldots,F_{i'_q-1}, F_{i'_q}\}.
	\]
 	Since $F_{i'_q}\cap F_{i'_q}= \emptyset $, it follows immediately that $F_{i'_q-1}\cap F_{i'_q}= \emptyset$ and $ F_{i'_q}$ forms a gap with the rest of the elements in $B$. This implies $\nu_0(\Gamma_t)=q+1$, a contradiction. Hence we conclude that $F_{i'_s}$ and $   F_{i_s}$ do not form a gap in $\Gamma_t$, and $i'_s+t =  i_s$.
 	Let $u_c= (u_b/m_{i'_s})   m_{i'_s+1}$. Then $u_b<u_c$ and $(u_c):(u_b)= (x_{i'_s+t, 0})$. Since $x_{i'_s+t, 0} \in F_{i_s}$, it follows that $(u_a):(u_b)\subset (u_c):(u_b)= (x_{i'_s+t, 0})$. It only remains to show that $u_c \in G(I^{[q]})$. Due to $F_{i'_1}< \ldots < F_{i'_s}<F_{i'_s+1}$, we conclude that  $F_{i'_s+1}$ does not intersect any facet in $\{F_{i'_1}, \ldots,  F_{i'_{s-1}}\}$. 	Since $i'_s+t =  i_s<\ldots < i_q$,  the facet $F_{i'_s+1} $ is disjoint with each of the facet in $\{F_{i'_{s+1}}, \ldots, F_{i'_q} \}=\{F_{i_{s+1}}, \ldots, F_{i_q} \}$. This shows that $u_c \in G(I^{[q]})$, as required. This completes the proof.
\end{proof}

   We point out that the linearity of the resolution concerning the $\nu_0$-th squarefree power of a $t$-path ideal needs to be specialized for path graphs, because it does not hold in general for broom graphs. We refer to \cite[Page 12]{EHHM} for a counter-example related to the 2-path ideal, or simply the edge ideal of broom graphs. 

		\section{Linearly related squarefree powers of simplicial trees}\label{Section: linearly related property}

         In this section we discuss the linearity of the first syzygy module of squarefree powers of the facet ideals attached to simplicial trees. We say that a graded ideal $I$, generated by homogeneous elements of degree $t$, is \textit{linearly related} if $\beta_{1,j}(I) = 0$ for all $j \neq 1 + t$. In particular, if $I$ is not linear related then $I$ does not have a linear resolution. A useful tool to investigate the linearly related property for a monomial ideal is provided by \cite[Corollary 2.2]{BHZ}, which we recall in Theorem \ref{Corollary 1.2}. 
        
 Let $I$ be a monomial ideal generated in degree $d$. In \cite{BHZ}, authors associated a graph $G_I$ to $I$ as follows: $V(G_I)=G(I)$, and  $\{u,v\} \in E(G_I)$ if and only if $\deg(\mathrm{lcm}(u,v))=d+1$. Moreover, for all $u,v\in G(I)$, the induced subgraph of $G_I$ on the vertex set $\{w\in V(G_I)\mid w\text{ divides } \lcm(u,v)\}$ is denoted by $G^{(u,v)}_{I}$. 
  
		\begin{Theorem}\cite[Corollary 2.2]{BHZ}\label{Corollary 1.2}
		   Let $I$ be a monomial ideal generated in degree d. Then $I$ is linearly related if and only if for all $u,v\in G(I)$ there is a path in $G^{(u,v)}_{I}$ connecting $u$ and $v$. 
		\end{Theorem}
	
		 In \cite[Lemma 5.2]{BHZ}, authors showed that if $I$ is the edge ideal of a simple graph, then $I^{[k]}$ is not linearly related for any $1 \leq k < \nu_0(G)$. This result cannot be extended even for the facet ideal of an arbitrary $2$-dimensional simplicial complex as observed in the following example. 
  
            \begin{Example}\em
   Let $\Delta$ be the simplicial complex whose facet ideal is
   \[I(\Delta)=(x_1x_2x_3, x_4x_5x_6, x_7x_8x_9, 
x_4x_5x_7,x_2x_4x_8,x_3x_5x_7, x_4x_8x_9, x_5x_6x_7,\]  
\[x_1x_4x_7, x_2x_5x_8, x_3x_6x_9, 
x_4x_7x_9, x_6x_7x_9, x_6x_8x_9, x_4x_6x_9) \]
The set $M=\{ \{1,2,3\}, \{4,5,6\}, \{7,8,9\}\}$ is a restricted matching of $\Delta$ because $\{1,2,3\}$ makes a gap with the rest of the facets in $M$. One can verify that there does not exists any restricted matching of $\Delta$ of size bigger than three.  It gives $\nu_0(\Delta)=3$. With \texttt{Macaulay2} \cite{M2}, we see that $I(\Delta)^{[2]}$ is linearly related.
		\end{Example}
  
Now, we prove an analogue of \cite[Lemma 5.2]{BHZ} for pure simplicial forests. Together with Proposition \ref{prop:difference of matchings}, it basically gives a necessary condition (stated in Corollary~\ref{cor:nec}) for the squarefree powers of the facet ideals of a pure simplicial forest to have a linear resolution.
  
		\begin{Theorem}\label{thm: k-squarefree power non-linearly related}
			Let $\Delta$ be a pure simplicial forest with $\dim(\Delta)>0$. Then $I(\Delta)^{[k]}$ is not linearly related and, hence, does not have a linear resolution for all $1 \leq  k < \nu_0(\Delta)$.
		\end{Theorem}
		
		\begin{proof}
			Let $k$ be an integer with $1 \leq  k < \nu_0(\Delta)$. In particular, this implies that $\nu_0(\Delta)\geq 2$.
            Let $\dim(\Delta)=n-1>0$ and $M=\{F_1,F_2,\ldots,F_{{\nu}_0(\Delta)}\}$  be a restricted matching of $\Delta$. Set $I=I(\Delta)$. Let $f_i=\prod_{j\in F_i}x_j$ and set $u=f_1f_2\cdots f_k$ and $v=f_2f_3\ldots f_{k+1}$ for $1 \leq k<{\nu}_0(\Delta)$ such that $F_{k+1}$ forms a gap with rest of the elements in $M$. By virtue of Theorem~\ref{Corollary 1.2} it is enough to show that $u$ and $v$ are disconnected in $G^{(u,v)}_{I^{[k]}}$.
			
			Note that $\{u,v\}$ is not an edge in $G^{(u,v)}_{I^{[k]}}$ because $\deg(\lcm(u,v))=kn+n>nk+1$. Suppose that $u$ and $v$ are connected in $G^{(u,v)}_{I^{[k]}}$. Then there exists some $w \in G(I^{[k]})$ such that $\{w,u\} $ is an edge in $G^{(u,v)}_{I^{[k]}}$ and $\deg(\lcm(u,w))= nk+1$. Since $w$ divides $\lcm(u,v)$ and $\deg(w)= nk$, there exist some $x \in F_{k+1}\subset \supp(v)$ and $y \in \supp (u)$ such that $\supp(w)=\{x\} \cup (\supp(u)\setminus \{y\})$. Let $w=w_1\cdots w_k$ with $G_i=\supp(w_i) \in \F(\Delta)$ for $i=1, \ldots, k$. After a relabelling of vertices, we may assume that $x \in G_1$ and $y\in F_1$. Let $G'_1=G_1\setminus\{x\}$ and $F'_1=F_1\setminus\{y\}$, and set
   $A=\{G'_1, G_2, \ldots, G_k\}$ and $B=\{F'_1, F_2, \ldots, F_k\}$. Observe that $G'_1 \not\subseteq F_\ell$, for all $\ell=2, \ldots, k$. Indeed, if  $G'_1 \subseteq F_\ell$ for some $\ell$, then $G_1$ belongs to the induced subcollection on $F_\ell \cup F_{k+1}$, which is a contradiction to the assumption that $F_{k+1}$ forms a gap with $F_\ell$.   
   
   We claim that $A=B$ and consequently $G'_1= F'_1$. To see this, we apply the similar argument as in the proof of Proposition~\ref{M=N}. Note that the elements of $A$ are pairwise disjoint and the elements of $B$ are pairwise disjoint as well. Moreover, the union of elements of $A$ coincides with the union of elements in $B$. Assume that $A\neq B$, and without loss of generality, we may assume that $A \cap B = \emptyset$. Consider the bipartite graph $H$ on the vertex set $A \cup B$ such that two vertices of $H$ are adjacent if and only if their intersection as facets of $\Delta$ is non-empty. Since $G'_1 \not\subseteq F_\ell$ for any $\ell$, we see that degree of $G'_1$ in $H$ is at least two. Moreover, due to $A \cap B = \emptyset$ we obtain that degrees of $G_2, \ldots G_k, F_2, \ldots, F_k$ are also at least two. Therefore, $H$ may have at most one vertex of degree one, namely, $F'_1$. This show that $H$ is not a forest and it contains a cycle. Let $C$ be the vertex set of a cycle in $H$. If $G'_1$ or $F'_1$ appears in $C$, then we replace them by $G_1$ and $F_1$, respectively. As argued in the proof of Proposition~\ref{M=N}, we see that the subcollection of $\Delta$ with facets in $C$ does not have a leaf, which is a contradiction to $\Delta$ being a simplicial tree. Therefore, $A=B$, and consequently $G_1\setminus\{x\}= F_1\setminus\{y\}$. This shows that $G_1$ belongs to the induced subcollection on $F_1 \cup F_{k+1}$, a contradiction to the assumption that $F_{k+1}$ forms a gap with $F_1$.

   From above argument we see that there does not exist any $w \in G^{(u,v)}_{I^{[k]}}$ adjacent to $u$. Therefore, $G^{(u,v)}_{I^{[k]}}$ is disconnected, as claimed.
\end{proof}

    \begin{Corollary}\label{cor:nec}
        Let $\Delta$ be a pure simplicial forest with $\dim(\Delta)>0$. If $I(\Delta)^{[k]}$ has a linear resolution then $k=\nu(\Delta)-1$ or $k=\nu(\Delta)$. 
    \end{Corollary}

    \begin{proof}
        It follows from Proposition \ref{prop:difference of matchings} and Theorem \ref{thm: k-squarefree power non-linearly related}. 	 \end{proof}
  	
   Next, we show that for any pure simplicial tree $\Delta$, if $I(\Delta)^{[k]}$ is linearly related then $I(\Delta)^{[k+1]}$ is also linearly related. Hence, if the highest squarefree power $I(\Delta)^{[\nu(\Delta)]}$ is not linearly related then $I(\Delta)^{[k]}$ cannot be linearly related for all $1\leq k\leq \nu(\Delta)$, in particular $I(\Delta)^{[k]}$ cannot have a linear resolution for all $1\leq k\leq \nu(\Delta)$. To show this, we first prove the following lemma. 
   
		\begin{Lemma}\label{lem:int}
			Let $\Delta$ be a simplicial tree. Further, let $M=\{F_1, \ldots, F_s\}$ and $N=\{G_1, \ldots, G_s\}$ be two $s$-matchings of $\Delta$. Then there exist $i,j \in \{1,\dots,s\}$ such that $F_i \cap G_k = \emptyset$ for all $k\neq j$.
		\end{Lemma}
		
		\begin{proof}
			If $F_i=G_j$ for any $i, j \in \{1,\dots,s\}$, then the assertion holds trivially. Assume that $F_i \neq G_j$, for all $i,j \in \{1,\dots,s\}$, that is, $M \cap N = \emptyset$. Let $H$ be the bipartite graph with $V(H)=M\cup N$ and $E(H)=\{ \{F_i, G_j\} : F_i \cap G_j \neq \emptyset, i,j \in \{1,\dots,s\}\}$. On contrary, assume that for each $i\in \{1,\dots,s\}$ there exist at least two $p,q\in\{1,\dots,s\}$ such that $F_i \cap G_p \neq \emptyset$ and $F_i \cap G_q \neq \emptyset$.  Then each vertex in $H$ has degree at least two. Therefore, $H$ contains an even cycle. After rearranging the indices, we may assume that $F_1, G_1, F_2, G_2, \ldots, G_t , F_{t+1}=F_1$ is a cycle of length $t$ in $H$. 
			
			Consider the subcollection $\Delta'= \langle F_1, \ldots, F_t, G_1, \ldots, G_t \rangle$. Since $M$ and $N$ are matching of $\Delta$, it follows that the sets $F_i \cap G_i$ and $G_i\cap  F_{i+1}$ are distinct for all $i=1, \ldots, t$. This shows that the subcollection $\Delta'= \langle F_1, \ldots, F_t, G_1, \ldots, G_t \rangle \subset \Delta$ has no leaf, and $\Delta$ is not a simplicial tree, a contradiction.
		\end{proof}
		
		\begin{Theorem}\label{theo:linearly related}
			Let $\Delta$ be a pure simplicial tree. If $I(\Delta)^{[k]}$ is linearly related then $I(\Delta)^{[k+1]}$ is also linearly related. 
		\end{Theorem}

		\begin{proof}
			Let $I=I(\Delta)$ and $u,v \in G(I^{[k+1]})$ with $u=f_1 \cdots f_{k+1}$ and $v=g_1 \cdots g_{k+1}$  and $f_1, \ldots, f_{k+1}, g_1, \ldots, g_{k+1} \in G(I)$. Following \cite[Corollary 1.2]{BHZ}, it is enough to show that $u$ and $v$ are connected by a path in $G_{I^{[k+1]}}^{(u,v)}$. Let $F_i=\supp(f_i)$ and $G_i=\supp(g_i)$, for all $i=1, \ldots, k+1$. Then $M=\{F_1, \ldots, F_{k+1}\}$ and $N=\{G_1, \ldots, G_{k+1}\}$ are $k$-matching of $\Delta$. 
			
			It follows from Lemma~\ref{lem:int} that there exist $i,j \in \{1,\dots,k+1\}$ such that $F_i \cap G_t = \emptyset$ for all $t\neq j$. After rearranging the indices, we may assume that $i=j=k+1$. Then $F_{k+1}\cap (G_1\cup \ldots \cup G_k)= \emptyset$. Since $I^{[k]}$ is linearly related, the monomials $ u'=f_1\cdots f_k$ and $v'=g_1  \cdots g_k$  are connected by a path, say $P_1:u'=w_0, w_1, \ldots, w_s=v'$, in $G_{I^{[k]}}^{(u',v')}$. Since $w_i$ divides $\lcm(u',v')$, we obtain $\supp(w_i) \cap F_{k+1}= \emptyset$ and $w_if_{k+1} \in G(I^{[k+1]})$, for all $i=0, \ldots, s$. Moreover, $w_i f_{k+1}$ divides $\lcm(u,v)$ and hence $w_if_{k+1} \in G_{I^{[k+1]}}^{(u,v)}$. This gives a path 
			\[
			Q_1: u=w_0f_{k+1}, w_1f_{k+1}, \ldots, w_sf_{k+1}=v'f_{k+1}
			\]
			in $G_{I^{[k+1]}}^{(u,v)}$.
			Proceeding in a similar way, we construct a path from $u''= g_2\cdots g_k f_{k+1}$ to $v''=g_2\cdots g_k g_{k+1}$ in $G_{I^{[k]}}^{(u'',v'')}$ which provides a path $Q_2$ from $g_1u''=v'f_{k+1}$ to $g_1v''=v$ in $G_{I^{[k+1]}}^{(u,v)}$. Joining $Q_1$ and $Q_2$ gives us a path connecting $u$ and $v$ in $G_{I^{[k+1]}}^{(u,v)}$, as required. 
		\end{proof}

         We conclude this section with a description of the degrees of the vanishing graded Betti numbers with homological degree one. An application of this result will be provided in Proposition \ref{prop: regularity (nu-1)-squarefree power of line graph} in order to give a lower bound for the regularity of the $(\nu-1)$-squarefree power of the facet ideal of the $t$-path simplicial tree of a path graph. To this end, we recall some definitions below.
        
		Let $P$ be a poset. The comparability graph of $P$, denoted by $G_P$, is a graph whose vertex set consists of the elements of $P$ and $\{a,b\} \in E(G_P)$ if and only if $a$ and $b$ are comparable in $P$.
		
		Let $I$ be a monomial ideal. The lcm-lattice of $I$, denoted by $L(I)$, is the poset whose elements are the least common multiples of subsets of monomials in $G(I)$ which are ordered by divisibility. By the definition, $L(I)$ has $1$ as the unique minimal element. For any $u \in L(I)$,  the induced subposet of $L(I)$ with elements $v \in L(I)$ such that $1 < v < u$, is denoted by the open interval $(1, u)$. The simplicial complex $\Delta((1, u))$ is the order complex of the poset $(1, u)$. 
		
		In the following theorem and in next section, we adopt the following notation to refer to $t$-path simplicial trees of path graphs. Let $P_n$ be the path graph on vertex $\{1, \ldots, n\}$ and edges $\{i, i+1\}$ for all $i=1, \ldots, n-1$. For any $t \leq n$, we denote the $t$-path simplicial tree of $P_n$ by $\Gamma_{n,t}$. Then 
  \[
\F(\Gamma_{n,t})=\{ F_i=\{i, i+1, \ldots, i+t-1\} : i=1, \ldots, n-t+1\}.  
  \]
The ideal $I_{n,t}= I(\Gamma_{n,t})$ is called the {\em $t$-path ideal of $P_n$}. We label the generators of $I_{n,t}$ as $f_1, \ldots, f_{n-t+1}$ such that $f_i= \prod_{j \in F_i}x_j$ for each $i$. Moreover, we write $f_i < f_j$, if $i<j$. Let $u,v \in G(I_{n,t}^{[k]})$ with $u=f_{i_1}\ldots f_{i_k}$ and $v=f_{j_1}\ldots, f_{j_k}$. Let $A_{(u,v)}$ be the set of indices such that $i_a \in A_{(u,v)}$ if and only if $f_{i_a} \neq f_{j_a}$. Now we are ready to prove the following. 
		\begin{Theorem}\label{theorem: betti number t-path ideal}
			Let $\Gamma_{n,t}$ be the $t$-path simplicial tree of a path graph $P_n$ and $I_{n,t}= I(\Gamma_{n,t})$. Then $\beta_{1,p}(I_{n,t}^{[k]}) = 0$ if $p \notin \{ kt+1,(k+1)t\}$.
		\end{Theorem}
		
		\begin{proof}
			Let $I=I_{n,t}$. Following the construction of the Taylor's complex, see \cite[Section 7.1]{HH1}, attached to $I^{[k]}$, we have $\beta_{1,m}(I^{[k]})=0$ if $m\neq \lcm(u_1,u_2)$ for any $u_1,u_2 \in G(I^{[k]})$. Moreover, using the result of Gasharov, Peeva and Welker \cite{GPW}, we have for all $i\geq 0$ and for all $ m \in L(I^{[k]})$
			\[
			\beta_{1,m} (I^{[k]})= \dim_K \tilde{H}_{0} (\Delta((1,m)) ;K).
			\]
			Recall that $\dim_K \tilde{H}_{0} (\Delta((1,m)) ;K)$ is $c-1$, where $c$ is the number of connected components of $\Delta((1,m))$ (see \cite[Chapter 1-Section 7]{MJ}). Note that the maximum degree that $m$ can have is $2kt$. This gives $\beta_{1,p}(I^{[k]}) = 0$ for all $p > 2kt$. Therefore, to prove the assertion, it is enough to show the following: if $m$ is the least common multiple of two elements in $G(I^{[k]})$, with $kt+1 < \deg(m)  \leq 2kt$ and $\deg(m)\neq (k+1)t$, then the open interval $(1,m)$ in the lcm-lattice $L(I^{[k]})$ is connected. Moreover, to show that $P=(1, m)$ is connected, it is enough to show that the comparability graph $G_P$ of $P$ is connected.  
           
            By the definition of $G_P$, the vertices of $G_P$ are those monomials in $L(I^{[k]})$ that are different from $1$ and strictly divide $m$. Therefore, any monomial in $V(G_P)$ with degree strictly greater than $kt$ is adjacent with some elements of degree $kt$, which are precisely the generators of $I^{[k]}$. Then, it is enough to show that for any $v,w \in V(G_P)$ with $\deg(v)=\deg(w)=kt$ and $v\neq w$, there is a path in $G_P$ that connects $v$ and $w$. If $\lcm(v,w) \in V(G_P)$, then $\lcm(v,w)$ is a common neighbor of $v$ and $w$ and we are done. Assume that $\lcm(v,w) \notin V(G_P)$. Since $v$ and $w$ strictly divide $m$, we note that $\lcm(v,w) \notin V(G_P)$ if and only if $\lcm(v,w) = m$. Let $v=f_{i_1} \ldots f_{i_k}$ and $w=	f_{j_1} \ldots f_{j_k}$ with $i_1<\ldots < i_k$ and $j_1< \ldots < j_k$. First we assume that $kt+1 < \deg (m) < (k+1)t$. Then $\deg(m)=kt+r$ for some $2 \leq r \leq t-1$. Since $v\neq w$, there exists some index $a$ for which $f_{i_a} \neq f_{j_a}$.  \medskip
			
			Case (1): Let $|A_{(v,w)}|=1$. This means there exists exactly one index $a$ for which $f_{i_a}\neq f_{j_a}$, and $f_{i_p} = f_{j_p}$ for all $p \neq a$. Without loss of generality, we may assume that $f_{i_a} < f_{j_a}$. Then $j_a=i_a+r$. Set $z_q:=f_{i_a+q} $ for $q=0, \ldots, r$. Consider the elements of $G(I^{[k]})$ for each $q=0, \ldots, r$ given by $v_q= (v/f_{i_a})z_q$. Then $v_0=v$, $v_r=w$ and $\deg(m_q)=kt+1< \deg (m)$ where  $m_q=\lcm(v_q, v_{q+1})$. This shows that $m_q$ strictly divide $m$ and  $m_q \in (1,m)$, that is, $m_q \in V(G_P)$. This gives us a path $v=v_0, m_0, v_1, \ldots, m_{r-1}, v_r=w$ in $G_P$ connecting $v$ and $w$, as required.
   \medskip
			
			Case (2): Let $|A_{(v,w)}|>1$, and $a$ be the smallest index for which $f_{i_a}\neq f_{j_a}$. Without loss of generality, we may assume that $f_{j_a}< f_{i_a}$. Then $\supp(f_{j_a}) \cap \supp(f_{i_p})= F_{j_a} \cap F_{i_p}= \emptyset$ for all $p \neq a$. Set $v_1=( v/f_{i_a})f_{j_a}$. It yields $v_1$ corresponds to a $k$-matching of $\Gamma_{n,t}$ and $v_1 \in G(I^{[k]})$. Moreover, $\deg(\lcm(v, v_1))=kt+\ell < \deg(m)$ with $\ell<r$ because $|A_{(v,w)}|>1$.  Hence, $\lcm(v, v_1)$ strictly divides $m$ and $\lcm(v, v_1) \in V(G_P)$. So far, we have a path $v, \lcm(v, v_1), v_1$ in $G_P$. Note that $|A_{(v_1,w)}|< |A_{(v,w)}|$. 
			
			We repeat our argument by replacing $v$ with $v_1$ to obtain $v_2$ such that $|A_{(v_2,w)}|<|A_{(v_1,w)}|< |A_{(v,w)}|$. After  $d=|A_{(v,w)}|-1$ number of steps, we obtain $v_d$ for which $|A_{(v_d,w)}|=1$. Then by repeating the arguments as in Case (1), we further obtain a path from $v_d$ to $w$ in $G_P$ which can be augmented with the path $v, \lcm(v, v_1), v_1$, $ \lcm(v_1, v_2), v_2, \ldots$, $\lcm(v_{d-1}, v_d), v_d$. In this way, we obtain a path from $v$ to $w$ in $G_P$. 
			
			Now assume that $(k+1)t < \deg (m) < 2kt$. In this case $|A_{(v,w)}|>2$.  We proceed as follows: let $v_1=v$, $w_1=w$, $s=r=1$ and perform the following step.
   \medskip
			
			Step-$s$: We set $v_s=f_{s_1}\cdots f_{s_k}$ with $s_1< \ldots < s_k$, and $w_r=f_{r_1}\cdots f_{r_k}$ with $r_1< \ldots < r_k$. Let $a$ be the smallest integer for which $f_{s_a}\neq f_{r_a}$. If $f_{r_a} < f_{s_a}$, then we construct $v_{s+1}$ as follows such that $\lcm(v_{s},v_{s+1}) \in V(G_P)$. 
			
			Set $v_{s+1}:=(v_{s}/f_{s_a})f_{r_a}$. Since $f_{s_1}= f_{r_1}$, \ldots, $f_{r_{a-1}}= f_{s_{a-1}}$. and $f_{r_a} < f_{s_a}<f_{s_{a+1}}< \ldots, f_{s_k}$, it immediately follows that $v_{s+1} \in G(I^{[k]})$. Moreover,
			\begin{itemize}
				\item[(i)]  if $\supp(f_{r_a}) \cap \supp(f_{s_a})=\emptyset$, then  $\deg(\lcm(v_{s},v_{s+1}))=(k+1)t < \deg(m)$. 
				
				\item[(ii)] if $\supp(f_{r_a}) \cap \supp(f_{s_a}) \neq \emptyset$, then $\deg(\lcm(v_s,v_{s+1}))=kt+r< \deg(m)$ where $r<t$.
			\end{itemize}
			In both cases (i) and (ii) above, the $\lcm(v_{s},v_{s+1})$ strictly divides $m$. Therefore,  $\lcm(v_{s},v_{s+1}) \in V(G_P)$. At the end of this step, we obtain a path $v_s, \lcm(v_s, v_{s+1}), v_{s+1}$ in $G_P$. Also, $|A_{(v_{s}, w_r)}|> |A_{(v_{s+1}, w_r)}| $. If $|A_{(v_{s+1}, w_r)}| >0$, we set $v_s:=v_{s+1}$ and repeat Step-$s$. Otherwise, $v_{s+1}=w_r$ and we obtain the desired path
			\[
			v=v_1, \lcm(v_1, v_2), v_2, \ldots, v_s, \lcm(v_{s}, w_r), w_r, \lcm(w_{r-1}, w_r), w_{r-1}, \ldots, w_1=w
			\]
			connecting $v$ and $w$ in $G_P$. 

   If $f_{r_a} > f_{s_a}$ then we terminate Step-$s$ and go to Step-$r$ to construct $w_{r+1}$ such that $\lcm(w_{r},w_{r+1}) \in V(G_P)$.
   \medskip
   
			Step-$r$: Set $w_{r+1}:=(w_{r}/f_{r_a})f_{s_a}$. From a similar discussion as in the construction of $v_{s+1}$ above, it immediately follows  that $w_{r+1} \in G(I^{[k]})$. Moreover, 
			\begin{itemize}
				\item[(i)]  if $\supp(f_{r_a}) \cap \supp(f_{s_a})=\emptyset$, then  $\deg(\lcm(w_{r},w_{r+1}))=(k+1)t < \deg(m)$. 
				
				\item[(ii)] if $\supp(f_{r_a}) \cap \supp(f_{s_a}) \neq \emptyset$, $\deg(\lcm(w_r,w_{r+1}))=kt+r< \deg(m)$ where $r<t$.
			\end{itemize}
			In both cases (i) and (ii) above, $\lcm(w_{r},w_{r+1})$ strictly divides $m$ and  $\lcm(w_{r},w_{r+1}) \in V(G_P)$. At the end of this step, we obtain a path $w_r, \lcm(w_r, w_{r+1}), w_{r+1}$ in $G_P$. Also, $|A_{(w_{r}, v_s)}|> |A_{(w_{r+1}, v_s)}| $. If $|A_{(w_{r+1}, v_s)}| > 1$, we set $w_r:=w_{r+1}$ and repeat Step-$s$. Otherwise, $w_{r+1}=v_s$ and we obtain the desired path connecting $v_1$ and $w_1$ in $G_P$. 
			
			Note that, we perform Step-$s$ and Step-$r$ only a finite number of time. Indeed, $|A_{(v_{s}, w_r)}|> |A_{(v_{s+1}, w_r)}| $ at the end of Step-$s$ and $|A_{(w_{r}, v_s)}|> |A_{(w_{r+1}, v_s)}| $ at the end of Step-$r$. This completes the proof. 
		\end{proof}

		\section{Regularity of $t$-path ideals of path graphs} \label{Section: Regularity of t-path ideals of path graph}

   In this section, we study the regularity of the squarefree powers of the $t$-path ideal $I_{n,t}$ of path graph $P_n$. As the main result of this section, in Theorem~\ref{thm: equality regularity t-path ideals of path}, we provide a combinatorial description of $\reg(R/I_{n,t}^{[k]})$ in terms of the induced matching number of $\Gamma_{n-kt,t}$. To prove our main theorem, we first establish some preliminary results.
   
   First, we set some notations in order to prove the next lemma. If $\Delta$ is a simplicial complex and $F \in \F(\Delta)$, then 
    \[
    \Delta\setminus F := \langle G :  G \in \F(\Delta) \text{ with } F \cap G = \emptyset \rangle.
    \] 
    Moreover, given a simplicial forest $\Delta$ with a good leaf order $F_r, \ldots, F_1$, we set $f_i= \prod_{x_j \in F_i}x_j$, $\Delta_i=\langle F_{i}, \ldots, F_r\rangle$ for each $i=1, \ldots, r$, and $J_i=(f_1, \ldots, f_{i-1})$, for $i=2, \ldots, r$. Furthermore, we set $J_1=(0)$.

	\begin{Lemma}\label{reduction}
		Let $\Delta$ be a simplicial forest with good leaf order $F_r, \ldots, F_1$ and $f_i= \prod_{x_j \in F_i}x_j$ for all $i=1, \ldots, r$. Then for all $1 \leq k \leq \nu(\Delta)$, we have
		\begin{enumerate}
			\item[\em{(1)}]  $I(\Delta)^{[k+1]} : (f_1) = I(\Delta \setminus F_1)^{[k]}$.
			\item[\em{(2)}]  $(I(\Delta_i)^{[k+1]}+J_i) : (f_i) = I(\Delta_i \setminus F_i)^{[k]}+(J_i:(f_i))$, for all $2\leq i \leq r$. 
			 \item[\em{(3)}] $(I(\Delta_i)^{[k+1]}+J_i) + (f_i) 	= I(\Delta_{i+1} )^{[k+1]}+J_{i+1}$, for all $1\leq i \leq r$.
		\end{enumerate}
	\end{Lemma}
	
	\begin{proof}
		(1) Let $J=   I(\Delta)^{[k+1]} : (f_1) $ and $g \in J$.  Then $gf_1  = h f_{i_1} \ldots f_{i_{k+1}}$ such that $\{F_{i_1}, \ldots, F_{i_{k+1}}\}$ is a $(k+1)$-matching of $\Delta$ and $h$ is a monomial. Since $F_1$ is a good leaf of $\Delta$, it yields that $F_1$ is a leaf of the subcollection $\langle F_1, F_{i_1} \ldots F_{i_{k+1}}\rangle$ and we may assume that $F_1\cap F_{i_j} \subseteq F_1\cap F_{i_1}$ for all $j=2, \ldots k+1$. Therefore, any element  in $F_1 \setminus F_{i_1}$ is not contained in $F_{i_j}$ for all $j=2, \ldots, k+1$. It yields $f_1 $ divides $h f_{i_1}$ and  $f_{i_2}\ldots f_{i_{k+1}} $ divide $g$. Since $F_{i_1}$ does not intersect any $F_{i_j}$ for all $j=2, \ldots, k+1$, it follows that $F_1$ also does not intersect any 
		$F_{i_j}$ for all $j=2, \ldots, k+1$. From this we conclude that $F_{i_2}, \ldots, F_{i_{k+1}} \in \Delta \setminus F_1$ and $ g \in   I(\Delta \setminus F_1)^{[k]}$. The inclusion  $I(\Delta \setminus F_1)^{[k]} \subseteq J$ is obvious since any facet in $\Delta \setminus F_1$ is disjoint with $F_1$.
  
		(2) We have $(I(\Delta_i)^{[k+1]}+J_i) : (f_i) = (I(\Delta_i)^{[k+1]}:f_i )+(J_i : (f_i)) $, for all $2\leq i \leq r$. Since $F_i$ is a good leaf of $\Delta_i$, applying a similar argument as in proof of statement in (1), we have  $(I(\Delta_i)^{[k+1]}:f_i )= I(\Delta_i \setminus F_i)^{[k]}$, which gives us the equality in (2).  
        
        (3) We have $J_i+(f_i)= J_{i+1}$. Any element in $G(I(\Delta_i)^{[k+1]})$ is of the form $u=f_{i_1}\cdots f_{i_{k+1}}$ such that $\{F_{i_1}, \ldots, F_{i_{k+1}}\}$ is a $(k+1)$-matching. Note that if $f_i$ divides $u$ then $u\in (f_i) \subset J_{i+1}$. On the other hand, if $f_i$ does not divide $u$ then $u\in  I(\Delta_{i+1})^{[k+1]}$. This completes the proof.
	\end{proof}

\begin{Remark}\label{Remark}
{\em
    Let $\Gamma_{n,t}$ be the $t$-path simplicial tree of the path graph $P_n$. Then
\begin{equation*}   \nu(\Gamma_{n,t})=\displaystyle\Big\lfloor \frac{n}{t}\Big\rfloor ,\; \nu_0(\Gamma_{n,t})=\displaystyle\Big\lfloor \frac{n-1}{t}\Big\rfloor, \;
\nu_1(\Gamma_{n,t})=\displaystyle\Big\lceil \frac{n-t+1}{t+1}\Big\rceil. 
\end{equation*}
Below, we provide a brief reasoning to justify the above equalities. As before, we set
\[
\F(\Gamma_{n,t})=\{ F_i=\{i, i+1, \ldots, i+t-1\} : i=1, \ldots, n-t+1\}  
  \] 
    \begin{enumerate}
        \item Let $n=qt+r$ where $q=\lfloor \frac{n}{t}\rfloor$ and $0\leq r\leq t-1$. Then, it is easy to see that matching number of $\Gamma_{n,t}=q$. Indeed, $M=\{F_1, F_{1+t}, \ldots, F_{1+(q-1)t}\}$ is a maximal matching of $\Gamma_{n,t}$. 
        \item Observe that two facets $F_i, F_j$ of $\Gamma_{n,t}$ with $i<j$ form a gap if and only if $i+t<j$. Indeed, $F_i \cap F_j = \emptyset$ if and only if $i+t \leq j$, and if $j= i+t$, then $F_{i+1}$ belongs to the induced subcollection on $F_i\cup F_j$. With this observation, we form a restricted matching $M$ of $\Gamma_{n,t}$ of maximal size as follows: take a maximal matching $N$ of $\Gamma_{n-(t+1),t}$ as described in (1), and set $M=N \cup \{ F_{n-t+1}\}$. In $M$, the facet $F_{n-t+1}$ forms a gap with all the facets in $N$.  This gives \[\nu_0(\Gamma_{n,t})= \nu(\Gamma_{n-(t+1),t})+1=\displaystyle\Big\lfloor \frac{n-(t+1)}{t}\Big\rfloor +1=\displaystyle\Big\lfloor \frac{n-1}{t}\Big\rfloor.
        \]
        \item It follows from the definition of induced matching that any two facets in an induced matching form a gap. Let $n=q'(t+1)+r'$ with $0 \leq r' \leq t$. Then following the explanation given above in (2), it is easy to see that if $r'< t$, then the set $\{ F_1, F_{1+(1+t)}, \ldots, F_{1+(q'-1)(t+1)}\}$ is an induced matching of $\Gamma_{n,t}$ of maximal size, and $\nu_1(\Gamma_{n,t})=q'$. On the other hand, if $r'=t$, then $\{ F_1, F_{2+t}, \ldots, F_{1+(q'-1)(t+1)}, F_{n-t+1}\}$ is an induced matching of $\Gamma_{n,t}$ of maximal size, and $\nu_1(\Gamma_{n,t})=q'+1$. This completes the argument. 
\end{enumerate}
In Example~\ref{ex:restricted matching}(1), a precise computation of $\nu(\Gamma_{12,3})$, $\nu_0(\Gamma_{12,3})$ and $\nu_2(\Gamma_{12,3})$ is given.
}
        \end{Remark}

    \begin{Lemma}\label{lem:regofpath}
		Let $\Gamma_{n,t}$ be the $t$-path simplicial tree of a path graph $P_n$. Let $I_{n,t}= I(\Gamma_{n,t})$. Then
		\[
		\reg\left( \frac{R}{I_{n,t}}\right) = (t-1)\nu_1 (\Gamma_{n,t}).
		\]
	\end{Lemma}
	\begin{proof}
		It follows immediately from \cite[Corollary 5.4]{BHaK} and Remark \ref{Remark}. 
	\end{proof}

     Our strategy to provide an upper bound for $\reg(R/I_{n,t}^{[k]})$ relies on repeatedly utilizing the following short exact sequence $$0 \rightarrow R/(I:f) \rightarrow  R/I \rightarrow  R/(I+(f) )\rightarrow 0, $$
    where $I$ is an appropriate monomial ideal and $f$ is an element of $I$ of degree $d$. In fact it is known from \cite[Lemma 2.10]{DHS} that $$\reg(R/I) \leq \max\{\reg(R/ (I:f))+d, \reg(I+(f))\}.$$
	
	\begin{Theorem}\label{thm: upper bound}
		Let $\Gamma_{n,t}$ be the $t$-path simplicial tree of a path graph $P_n$. Let $I_{n,t}= I(\Gamma_{n,t})$. Then for any $1\leq k+1 \leq \nu(\Gamma_{n,t})$, we  have
		\[
		\reg\left( \frac{R}{I_{n,t}^{[k+1]}}\right)
		\leq
		kt+\reg  \left( \frac{R}{I_{n-kt,t}}\right) = kt+(t-1)\nu_1 (\Gamma_{n-kt,t}).
		\]
	\end{Theorem}
	\begin{proof} 
 We prove the assertion by applying induction on $n$. It is easy to see that the assertion holds for $P_t$.  Assume that the assertion holds for all paths on $m$ vertices such that $t \leq m<n$, that is, for each $1\leq k+1 \leq \nu(\Gamma_{m,t})$, we have 
 \[		\reg\left( \frac{R}{I_{m,t}^{[k+1]}}\right)
		\leq kt+(t-1)\nu_1 (\Gamma_{m-kt,t}).
		\]

		Given any two positive integers $i\leq j$, we set $[i,j]=\{i, i+1, \ldots, j\}$. To simplify the notation in this proof, we denote the path graph on vertices $[i,j]$  by $P_{[i,j]}$. In particular, for the path graph on the vertices $\{1, \ldots, n\}$, instead of writing $P_n$, we write $P_{[1,n]}$. Since $t$ is fixed throughout the proof, we denote the $t$-path ideal of $P_{[i,j]}$ simply by $I_{[i,j]}$, and $\Gamma_{n,t}$ by $\Gamma_{n}$. 
  
        Set $F_i=\{i,i+1, \ldots, i+t-1\}$ and $f_i= \prod_{j \in F_i}x_j$ for all $i=1, \ldots, n-t+1$. The order of the facets of $\Gamma_n$ given by $ F_{n-t+1}, \ldots, F_1$ is a good leaf order. Consider the following exact sequence 
		\[
		0 \rightarrow \frac{R}{I_{[1,n]}^{[k+1]} : f_{1}}
		\xrightarrow{f_1} \frac{R}{I_{[1,n]}^{[k+1]}} 
		\rightarrow  \frac{R}{I_{[1,n]}^{[k+1]}+(f_{1})}
		\rightarrow 0.
		\]
		Using Lemma~\ref{reduction}, we obtain 
		\[
		0 \rightarrow \frac{R}{I_{[1+t,n]}^{[k]} }
		\xrightarrow{f_1} \frac{R}{I_{[1,n]}^{[k+1]}} 
		\rightarrow  \frac{R}{I_{[2,n]}^{[k+1]}+I_{[1,t]}}
		\rightarrow 0 
		\]
		which gives
		\[
		\reg\left( \frac{R}{I_{[1,n]}^{[k]}} \right)   \leq   	\max
		\left\{ 
		t+\reg\left( \frac{R}{I_{[1+t,n]}^{[k]} } \right),  
		\reg\left(  \frac{R}{I_{[2,n]}^{[k+1]}+I_{[1,t]}} \right) 
		\right\}.
		\]
		Now we investigate $\reg( R/(I_{[2,n]}^{[k+1]}+I_{[1,t]}))$. Let $a$ be the maximum integer for which $k+1$ is the matching number of the path graph $P_{[a+1,n]}$. Indeed, it can be verified that $a=n-t(k+1)$. Using Lemma~\ref{reduction}, for any $2 \leq i\leq a$, we obtain 
		\[
		0 \rightarrow   \frac{R}{I_{[i+t,n]}^{[k]}+(I_{[1,t+i-2]} :f_i)}
		\xrightarrow{f_i} \frac{R}{I_{[i,n]}^{[k+1]}+I_{[1,t+i-2]}} 
		\rightarrow  \frac{R}{I_{[i+1,n]}^{[k+1]}+I_{[1,t+i-1]}} 
		\rightarrow 0 .
		\]
		Furthermore, since $\nu(P_{[a+1,n]})=k+1$, we get:
		\[
		0 \rightarrow   \frac{R}{I_{[a+1+t,n]}^{[k]}+(I_{[1,t+a-1]} :f_{a+1})}
		\xrightarrow{f_{a+1}}  \frac{R}{I_{[a+1,n]}^{[k+1]}+I_{[1,t+a-1]}} 
		\rightarrow  \frac{R}{I_{[1,t+a]}}
		\rightarrow 0 .
		\]
		Therefore,  
		\[
		\reg\left( \frac{R}{I_{n,t}^{[k+1]}}\right)   \leq   	\max \left\{ 
		t+\reg\left( \frac{R}{I_{[1+t,n]}^{[k]} } \right),  
		\reg \left( \frac{R}{I_{[1,t+a]}} \right), 
		\alpha 
		\right\}
		\]
		where 
		\[
		\alpha=\max_{2 \leq i \leq a+1}
		\left\{  
		t+\reg \left(  \frac{R}{I_{[i+t,n]}^{[k]}+(I_{[1,t+i-2]}:f_i)} \right)
		\right \}.
		\]
		Since $a=n-t(k+1)$, using Lemma~\ref{lem:regofpath} we obtain
  \[
	 \reg \left( \frac{R}{I_{[1,t+a]}} \right)=
  (t-1)\nu_1(\Gamma_{t+a})=
  (t-1)\nu_1(\Gamma_{n-tk})	
  \leq
  kt+(t-1)\nu_1 (\Gamma_{n-kt}).
		\]
		
		Note that $I_{[1+t,n]}^{[k]} $ can be identified as the $t$-path ideal of the path graph $P_{[1,n-t]}$, using the induction hypothesis, we obtain 
		\[
		t+\reg\left( \frac{R}{I_{[1+t,n]}^{[k]} } \right)\leq t+(k-1)t+(t-1)\nu_1 (\Gamma_{n-t-(k-1)t})=
		kt+(t-1)\nu_1 (\Gamma_{n-kt}).
		\]
		
		Now we analyze $\alpha$ and compare it with $kt+(t-1)\nu_1 (\Gamma_{n-kt})$. For any $2 \leq i\leq a+1$, the ideals $I_{[i+t,n]}^{[k]}$ and $I_{[1,t+i-2]} :f_i$ lie in disjoint set of vertices. Therefore
		\begin{equation}
			\begin{split}
				\reg \left(  \frac{R}{I_{[i+t,n]}^{[k]}+(I_{[1,t+i-2]} :f_i)} \right)
				&= \reg \left(  \frac{R}{I_{[i+t,n]}^{[k]}} \right)+\reg \left(  \frac{R}{(I_{[1,t+i-2]} :f_i)} \right).
			\end{split}
		\end{equation}
		
    Since $	I_{[1,t+i-2]}= (f_1, \ldots, f_{i-1}) $, for $2 \leq i\leq t$, we have
		\begin{equation}\label{eq:colon}
			\begin{split}
				I_{[1,t+i-2]} : f_i &= (f_1, \ldots, f_{i-1}) : f_i \\
				&=(f_1, \ldots, f_{i-t-1}) : f_i+(f_{i-t}, \ldots, f_{i-1}) : f_i\\
				&= (f_1, \ldots, f_{i-t-1})+(x_{i-1})= I_{[1,i-2]}+(x_{i-1}).
			\end{split}
		\end{equation}
		If $i \leq t+1$, then we set $(f_1, \ldots, f_{i-t-1})=0$ in above equation. By identifying $I_{[i+t,n]}^{[k]}$ as $t$-path ideal of the path graph $P_{[1, n-t-i+1]}$, and using induction hypothesis yields 
	  \[
		\reg \left(  \frac{R}{I_{[i+t,n]}^{[k]}} \right) 
		\leq 
		(k-1)t+(t-1)\nu_1(\Gamma_{n-i+1-kt}).
	\]
	On the other hand using (\ref{eq:colon}) and Lemma~\ref{lem:regofpath} we have 
		
		\[
		\reg \left(  \frac{R}{(I_{[1,t+i-2]} :f_i)} \right) 
		\leq 
		(t-1)\nu_1(\Gamma_{i-2}).
		\]
		Since $\nu_1(\Gamma_{i-2})+\nu_1(\Gamma_{n-i+1-kt}) \leq \nu_1(\Gamma_{n-1-kt})$, we see that $\alpha$ is less than $	kt+(t-1)\nu_1 (\Gamma_{n-kt})$. This completes the proof.
	\end{proof}

     Next we show that the upper bound given in Theorem~\ref{thm: upper bound}, is indeed equal to $\reg\left( \frac{R}{I_{n,t}^{[k+1]}}\right)$. To do this, we first compute the regularity of $(\nu(\Gamma_{n,t})-1)$-th power of $I_{n,t}$.
 
  \begin{Proposition}
		\label{prop: regularity (nu-1)-squarefree power of line graph}
		Let $\Gamma_{n,t}$ be the $t$-path simplicial tree of path graph $P_n$, and $I_{n,t}= I(\Gamma_{n,t})$. Suppose that $\nu(\Gamma_{n,t}) -1\neq \nu_0(\Gamma_{n,t})$ and $\nu(\Gamma_{n,t}) >1$. Then $$\reg \left( \frac{R}{I_{n,t}^{[\nu(\Gamma_{n,t})-1]}}\right)= \nu(\Gamma_{n,t}) t-2.$$
	\end{Proposition}
	\begin{proof}
		To simplify the notation, we denote $\nu(\Gamma_{n,t})$ and $\nu_0(\Gamma_{n,t})$ by $\nu$ and $\nu_0$, respectively. Since $\reg(R/I_{n,t}^{[\nu-1]})= \reg (I_{n,t}^{[\nu-1]})-1$, it is enough to show that $\reg (I_{n,t}^{[\nu-1]})=\nu t-1$. 
  
  The assumption $\nu-1\neq \nu_0$ together with Proposition~\ref{prop:difference of matchings} gives that $\nu=\nu_0$. Then from Theorem~\ref{thm: k-squarefree power non-linearly related}, we see that $I_{n,t}^{[\nu-1]}$ is not linearly related. Thanks to  Theorem~\ref{theorem: betti number t-path ideal}, we obtain $\beta_{1,\nu t} (I_{n,t}^{[\nu-1]})\neq 0$. Therefore, $\reg (I_{n,t}^{[\nu-1]})\geq  \nu t-1$. Now, we show that $\reg (I_{n,t}^{[\nu-1]})\leq  \nu t-1$.  From Theorem \ref{thm: upper bound} we have 
		\[
		\reg (I_{n,t}^{[\nu-1]})\leq (\nu -2)t+(t-1)\nu_1(\Gamma_{n-(\nu-2)t,t}))+1.
		\]
		It is sufficient to prove that $\nu_1(\Gamma_{n-(\nu-2)t,t})=2$.  There exists some $j$ such that $n=\nu t+j$ with $1\leq j\leq t-1$. Indeed, if $j\geq t$, then there exists a matching whose cardinality is strictly more that $\nu$, which is a contradiction with being $\nu$ the matching number. If $j=0$ then $n=\nu t$, so $\nu_0<\nu$, which is a contradiction with being $\nu=\nu_0$.  Since $n-(\nu-2)t=2t+j$, we obtain $\nu_1(\Gamma_{n-(\nu-2)t,t})=2$, due to Remark~\ref{Remark}, as desired.
	\end{proof}	

    \begin{Corollary}\label{coro: nu-1 squarefree}
    Let $\Gamma_{n,t}$ be the $t$-path simplicial tree of a path graph $P_n$ and $I_{n,t}= I(\Gamma_{n,t})$. Denote the matching number $\nu(\Gamma_{n,t})$ and the restricted matching number $\nu_0(\Gamma_{n,t})$ of $\Gamma_{n,t}$ by $\nu$ and $\nu_0$, respectively. Suppose that $\nu>1$.  Then 
		\[
		\reg\left( \frac{R}{I_{n,t}^{[\nu-1]}}\right)=(\nu-2)t+(t-1)\nu_1 (\Gamma_{n-(\nu-2)t,t})=
		\begin{cases} 
			t\nu- 1 & \text{ if } \nu-1=\nu_0 \\
			t\nu- 2 & \text{ if } \nu-1\neq \nu_0
		\end{cases}
		\]        
    \end{Corollary}

    \begin{proof}
        The claim follows from Theorem \ref{thm:broomgraph-linearquotient} and Proposition \ref{prop: regularity (nu-1)-squarefree power of line graph}. 
    \end{proof}

 	  In order to prove a lower bound for $\reg(R/I_{n,t}^{[k]})$, for $1\leq k+1<\nu(\Gamma_{n,t})$, we need to introduce an equivalent form of Hochster's formula (\cite[Theorem 2.8]{AF}). Let  $\Delta$ be a simplicial complex. An orientation on $\Delta$ is a linear order $<$ on the vertex set of $\Delta$. In such a case, $(\Delta,<)$ is said to be an \textit{oriented} simplicial complex. Let $\Delta$ be a $d$-dimensional oriented simplicial complex with vertex set $V(\Delta)$ and $i\in \{1,\dots,d\}$. A face $F=\{v_1,\dots, v_{i+1}\}$, with $v_1<\dots <v_{i+1}$, is said to be \textit{oriented}; in such a case, we write $F=[v_1,\dots, v_{i+1}]$. We denote by $C_i(\Delta)$ the $K$-vector space generated by all the oriented $i$-dimensional faces of $\Delta$. 
 	  The augmented oriented chain complex of $\Delta$ is the complex 
 	\[
 	0 \xrightarrow[]{ \ \ } C_d(\Delta) \xrightarrow[]{ \ \delta_d\ } C_{d-1}(\Delta) \xrightarrow[]{ \delta_{d-1}} \dots \xrightarrow[]{ \ \ } C_1(\Delta) \xrightarrow[]{\ \delta_1\ }  C_0(\Delta) \xrightarrow[]{\ \delta_0\ } K \xrightarrow[]{ \ \ } 0, 
 	\]
 	where $\delta_0(v)=1$ for all $v\in V(\Delta)$ and, for any $1\leq i \leq d$, the map $\delta_i : \; C_i(\Delta) \rightarrow C_{i-1}(\Delta)$ acts on basis elements as follows: 
 	\[
 	\delta_i \left( 
 	\left[
 	v_{1},\ldots, v_j, \ldots, v_{i+1}
 	\right]\right)
 	=
 	\sum_{j=1}^{i+1} (-1)^{j+1}
 	\left[
 	v_1, \ldots, \hat{v_j}, \ldots, v_{i+1}
 	\right],
 	\]
 	where $\hat{v_j}$ means that $v_j$ is removed. Recall that the $i$-th reduced simplicial homology group of $\Delta$ over $K$ is defined as $\tilde{H}_{i} ( \Delta; K)=\ker \delta_{i} / \Im \delta_{i+1}$. For convention, set $\tilde{H}_{-1}(\emptyset;K)=K$ and $\tilde{H}_{i}(\emptyset;K)=0$, for all $i\geq 0$. Moreover, it is well-known that, if $\Delta\neq\emptyset$, then $ \dim_K\tilde{H}_{0}(\Delta;K)$ is one less than the number of the connected components of $\Delta$ (see \cite[Chapter 1-Section 7]{MJ}).
 
Recall that, for any $Y\subseteq V$, an
induced subcollection of $\Delta$ on $Y$, denoted by $\Delta_Y$, is the simplicial complex whose vertex set is a subset of $Y$ and the facet set is $\{F \in \mathcal{F}(\Delta):F \subseteq Y\}$. If $\Delta=\langle F_1,\dots,F_s\rangle$, then we define the complement of a face $F$ of $\Delta$ in $Y$ to be $F^c_Y=Y\setminus F$ and the complement of $\Delta$ in $Y$ as $\Delta^c_Y=\langle (F_1)^c_Y,\dots,(F_s)^c_Y\rangle$. 

From \cite[Theorem 2.8]{AF} we know that, if $I$ is a  squarefree monomial ideal generated in single degree in $K[x_1,\dots,x_n]$, then
\begin{equation}\label{eq:hochster}
	\beta_{i,d} (I)=\sum_{\substack{\Gamma \subseteq \Delta(I)\\ \vert V(\Gamma)\vert =d}} \dim_K \tilde{H}_{i-1} (\Gamma^c_{V(\Gamma)}),
\end{equation}
	where the sum is taken over the induced subcollections $\Gamma$ of the facet complex $\Delta(I)$ which have $d$ vertices.

 \begin{Theorem}\label{thm: lower bound}
		Let $\Gamma_{n,t}$ be the $t$-path simplicial tree of a path graph $P_n$. Let $I_{n,t}= I(\Gamma_{n,t})$. Then for any $1\leq k+1 \leq \nu(\Gamma_{n,t})$, we  have
		\[
		kt+(t-1)\nu_1 (\Gamma_{n-kt,t}) \leq \reg\left( \frac{R}{I_{n,t}^{[k+1]}}\right)	
		\]
	\end{Theorem}

        \begin{proof}
          By replacing $k+1$ with $k$, it is equivalent to show for any $1\leq k \leq \nu(\Gamma_{n,t})$, the following inequality 
            \begin{equation}\label{ineq}
		(k-1)t+(t-1)\nu_1 (\Gamma_{n-(k-1)t,t}) \leq \reg\left( \frac{R}{I_{n,t}^{[k]}}\right).	
		\end{equation}
   Due to Lemma~\ref{lem:regofpath}, it is enough to consider the case when $k\geq 2$. Fix $k,t\in \NN$ with $k,t\geq 2$. We divide the discussion into distinct cases depending on the value of $n$. First, observe that $kt \leq n$ due to Remark~\ref{Remark} and the assumption $k\leq \nu(\Gamma_{n,t})$. 

   First, we consider the case when $kt\leq n \leq  kt+t$. If $kt\leq n < kt+t$, then $\nu(\Gamma_{n,t})=k$ and thanks to Theorem~\ref{thm:broomgraph-linearquotient} (i) we obtain $\reg(R/I_{n,t}^{[k]})=kt-1$.
 If $n=kt+t$, then $\nu(\Gamma_{n,t})=k+1$, and  due to Corollary \ref{coro: nu-1 squarefree}, we get $\reg(R/I_{n,t}^{[k]})=kt-1$. On the other hand, if $kt\leq n \leq kt+t$, then due to Remark~\ref{Remark}, we have 
$\nu_1 (\Gamma_{n-(k-1)t,t}) =1$ which gives $(k-1)t+(t-1)\nu_1 (\Gamma_{n-(k-1)t,t}) =kt-1$. Hence the inequality in (\ref{ineq}) holds when $kt\leq n \leq kt+t$. 

Now, let $kt+j(t+1) \leq n\leq kt+j(t+1)+t$, for some  $j\geq 1$. From here on, we separate the discussion into two cases, namely Case(1): $n=kt+j(t+1)$, and Case(2): $kt+j(t+1) < n\leq kt+j(t+1)+t$. To simplify the discussion, we will argue on $\reg (I_{n,t}^{[k]})$ and use $\reg(R/I_{n,t}^{[k]})= \reg (I_{n,t}^{[k]})-1$ to make the final conclusion. 
\medskip
     
     Case (1): Let  $n=kt+j(t+1)$, for some $j$. Hereafter, we set $n_0:=n$. Let $\Delta(I_{n_0,t}^{[k]})$ be the simplicial complex whose facet ideal is $I_{n_0,t}^{[k]}$. We show that $\beta_{i_0,n_0} (I_{n_0,t}^{[k]})\neq 0$ for $i_0=n_0-[(k-1)t+(t-1)\nu_1(\Gamma_{j(t+1)+t,t})+1]$. Since $\nu_1(\Gamma_{j(t+1)+t,t})=j+1$, we obtain after simplifying that $i_0=2j$. Using (\ref{eq:hochster}) gives
		\[
	\beta_{i_0,n_0} (I_{n_0,t}^{[k]})= \dim_K \tilde{H}_{i_0-1} ((\Delta(I_{n_0,t}^{[k]}))^c).
	\]
    We choose a subset $Q$ of $\{1,\dots,n_0\}$ with  $\vert Q\vert =i_0+1=2j+1$ as follows:
		\[
		Q=\{kt+h(t+1),kt+h(t+1)+1:h=0,\dots,j-1\}\cup \{n_0\}.
		\]
		Consider
		\[\gamma=\sum_{r=1}^{i_0+1} (-1)^{r+1}
		\left[
		v_1, \ldots, \hat{v_r}, \ldots, v_{i_0+1}
		\right],\]
		where $v_p\in Q$ for all $p=1,\dots,i_0+1$ and $v_p<v_{p+1}$ for all $p=1,\dots,i_0$. First, we prove that $\gamma\in \ker \delta_{i_0-1}$. In order to make sense that $\delta_{i_0-1}$ acts on $\gamma$, we need $\gamma \in C_{i_0-1}((\Delta(I_{n_0,t}^{[k]}))^c)$. We begin by proving $$\left[
		v_1, \ldots, \hat{v_r}, \ldots, v_{i_0+1}
		\right]\in C_{i_0-1}((\Delta(I_{n_0,t}^{[k]}))^c)$$ for all $r=1,\dots,i_0+1$. We consider the following three cases.
		\begin{itemize}
			\item If $r=1$, then $\left[
			\hat{v_1}, {v_2}, \ldots, v_{i_0+1}
			\right]$ belongs to $C_{i_0-1}((\Delta(I_{n_0,t}^{[k]}))^c)$ because $v_1=kt$ and $\{1,\dots,kt\}$ is a facet of $\Delta(I_{n_0,t}^{[k]})$. 
            \item If $r=i_0+1$, then $\left[{v_1}, \ldots, v_{i_0},\hat{v}_{i_0+1}	\right]\in C_{i_0-1}((\Delta(I_{n_0,t}^{[k]}))^c)$ because $v_{i_0+1}= n_0$ and $\{1,\dots,(k-1)t\}\cup\{n_0-t+1,\dots,n_0\}$ is a facet of $ \Delta(I_{n_0,t}^{[k]})$.            
             \item 
             Now, let $1<r<i_0+1$. If $v_r=kt+h(t+1)$ for some $h\in \{1,\dots,j-1\}$. Then
                \[v_{r-1}=kt+(h-1)(t+1)+1= v_r-t,\]
                and $\{v_{r-1}+1,\dots,v_r\}$ is a path on $t$ vertices. Therefore $\left[v_1, \ldots, \hat{v_r}, \ldots, v_{i_0+1} \right]\in C_{i_0-1}((\Delta(I_{n_0,t}^{[k]}))^c)$ because $\{1,\dots,(k-1)t\}\cup\{v_{r-1}+1,\dots,v_r\}$ is a facet of $\Delta(I_{n_0,t}^{[k]})$. 
                
                 The case when $v_r=kt+h(t+1)+1$, for some $h\in \{0,\dots,j-1\}$, can be argued in a  similar way. In fact, in this case we have  $v_{r+1}=kt+(h+1)(t+1)= v_r+t$. Then $\left[v_1, \ldots, \hat{v_r}, \ldots, v_{i_0+1} \right]\in C_{i_0-1}((\Delta(I_{n_0,t}^{[k]}))^c)$ since $\{1,\dots,(k-1)t\} \cup\{v_{r},\dots,v_{r+1}-1\}$ is a facet of $\Delta(I_{n_0,t}^{[k]})$. 
\end{itemize}
		Therefore $\gamma \in C_{i_0-1}((\Delta(I_{n_0,t}^{[k]}))^c)$. Moreover, 
    \begin{align*}
        &\delta_{i_0-1}(\gamma)
            = \sum_{r=1}^{i_0+1} (-1)^{r+1}
		\delta_{i_0-1}(\left[ v_1, \ldots, \hat{v_r},\dots,v_{i_0+1} \right])=\\
	&	= \sum_{r=1,\dots,i_0+1} (-1)^{r+1}
		\left(\sum_{\substack{s =1,\dots,i_0\\
		    	s\neq r}}(-1)^{s+1}
		\left[ v_1, \ldots, \hat{v_r},\dots, \hat{v_s} \ldots, v_{i_0+1} \right]\right).
    \end{align*}
	
	It is easy to check that, for fixed $r$ and $s$ the coefficient of $[v_1,\dots, \hat{v_r},\dots,\hat{v_s},\dots,v_{i_0+1}]$ is $(-1)^{r+s+2}-(-1)^{r+s+1}=0$. Hence $\gamma\in \ker \delta_{i_0-1}$.

 	We claim that $\gamma$ does not belong to $\Im\delta_{i_0}$. Indeed, if this were the case, we would have $\tilde{H}_{i_0-1} ( (\Delta(I_{n_0,t}^{[k]}))^c ; K)\neq 0$. Consequently, this would imply that $\beta_{i_0,n_0} (I_{n_0,t}^{[k]})\neq 0$, leading us to the inequality $(k-1)t+(t-1)\nu_1 (\Gamma_{n_0-(k-1)t,t}) \leq \reg (I_{n_0,t}^{[k]})$, as required.\\ 	
 	We now proceed to prove that $\gamma$ is not in $\Im \delta_{i_0}$, that is, there does not exist any element in $C_{i_0}((\Delta(I_{n_0,t}^{[k]}))^c)$ whose image under the boundary map $\delta_{i_0}$ is $\gamma$. We begin by noting that $[v_1,\dots,v_{i_0+1}]\not\in C_{i_0}((\Delta(I_{n_0,t}^{[k]}))^c)$ since, by the definition of $Q$, there does not exist a $k$-matching $S=\{S_1,\dots, S_k\}$ in the induced subgraph of $P_{n_0}$ on $\{1,\dots,n_0\}\setminus Q$ such that $\cup_{i=1}^k S_i$ is a facet of $\Delta(I_{n_0,t}^{[k]})$. Next, we define $\mathcal{B}_{i_0-1}$ as the basis of $C_{i_0-1}((\Delta(I_{n_0,t}^{[k]}))^c)$, consisting of all oriented $(i_0-1)$-dimensional faces of $(\Delta(I_{n_0,t}^{[k]}))^c$. Similarly, $\mathcal{B}_{i_0}$ represents the basis of $C_{i_0}((\Delta(I_{n_0,t}^{[k]}))^c)$. Let $A$ be the transformation matrix of $\delta_{i_0}$. Recall that each column of $A$ represents the coordinate vector of an element in $\mathcal{B}_{i_0}$ expressed in terms of the basis $\mathcal{B}_{i_0-1}$, consequently the entries of $A$ are either $1,-1$ and $0$. Furthermore, let $\textbf{x}$ represent the column vector of the components of an element in $(\Delta(I_{n_0,t}^{[k]}))^c$ with respect to $\mathcal{B}_{i_0}$. With this notation, we can write $\delta_{i_0}(\mathbf{x})=A\mathbf{x}$. Let $\mathbf{b}_\gamma$ be the column vector of the components of $\gamma$ with respect to $\mathcal{B}_{i_0-1}$. We aim to show that the linear system $A \mathbf{x} = \mathbf{b}_\gamma$ has no solutions. In other words, by Rouch\'e-Capelli Theorem, the ranks of the coefficient matrix $A$ and the augmented matrix $(A \vert \mathbf{b}_\gamma)$ are different. Let $R$ be the row of $(A\vert  \mathbf{b}_\gamma)$ corresponding to $[v_2, \dots, v_{i_0+1}]$. The row $R$ has exactly $(t-1)\frac{i_0}{2} + 1$ non-zero entries; specifically, in $R$ there is: 
 \begin{itemize}
 	\item $1$ in the column corresponding to $\mathbf{b}_\gamma$;
 	\item $-1$ in the column, denoted by $C_{j,a}$, which corresponds to the element of $\mathcal{B}_{i_0}$ of the form $[v_2, \dots,v_{j},\dots, a,\dots, v_{j+1}, \dots, v_{i_0+1}]$, for all even $j$ in $\{2,\dots,i_0\}$ and for all $a\in \NN$ with $v_{j} < a < v_{j+1}$; 
 	\item $0$ in the remaining columns.
 \end{itemize}
For example, if we consider the second squarefree power of the $3$-path ideal of the path graph $P_{14}$, we have $Q = \{6,7,10,11,14\}$ and $\gamma = [7,10,11,14] - [6,10,11,14] - [6,7,11,14] + [6,7,10,14] - [6,7,10,11]$. In the augmented matrix $(A \vert \mathbf{b}_\gamma)$ of $\delta_{4}$, the five non-zero entries of $R$ are given in Table \ref{Example row R}, specifically: $1$ appears in the column corresponding to $\mathbf{b}_\gamma$, and $-1$ appears in the columns corresponding to the elements $[7,8,10,11,14]$, $[7,9,10,11,14]$, $[7,10,11,12,14]$, and $[7,10,11,13,14]$, denoted by $C_{2,8}$, $C_{2,9}$, $C_{4,12}$, and $C_{4,13}$, respectively. Moreover, $0$ appears in the remaining entries of $R$, which are not illustrated due to the large size of the matrix.

\begin{table}[h!]
	\centering
	\footnotesize
	\scalebox{0.9}{
		\begin{tabular}{c|c|c|c|c|c}
			& 7 8 10 11 14 & 7 9 10 11 14 & 7 10 11 12 14 & 7 10 11 13 14 & $\mathbf{b}_\gamma$     \\ \hline
			7 10 11 14  & -1           & -1           & -1            & -1            & 1                     
		\end{tabular}
	}
	\caption{The non-zero entries of $R$ in $(A \vert \mathbf{b}_\gamma)$.}
	\label{Example row R}	
\end{table}

 To prove our claim, it is enough to show that we can perform elementary row operations within the augmented matrix $(A \vert \mathbf{b}_\gamma)$ to reduce $R$ to a row with zero entries everywhere except for a $1$ in the column corresponding to $\mathbf{b}_\gamma$. This transformation can be achieved by utilizing only certain specific rows of $(A \vert \mathbf{b}_\gamma)$, which will be delineated in the following description. For all $s\in [i_0/2]$, let $R(j_1,\dots,j_s;a_1,\dots,a_s)$ be the row of $(A \vert \mathbf{b}_\gamma)$ corresponding to the element of $\mathcal{B}_{i_0}$ of the form $$[v_2,\dots,v_{j_1},\dots,a_1,\dots, \widehat{v_{j_1+1}},\dots, v_{j_s-1},v_{j_s}, \dots,a_s,\dots,\widehat{v_{j_s+1}},\dots,v_{i_0+1}],$$	for all even $j_1,\dots,j_s$ in $\{2, \dots, i_0\}$ with $j_1< \dots <j_s$ and $a_1,\dots,a_s\in \NN$ with $v_{j_h} <a_h < v_{j_h+1}$ for each $h\in [s]$. Hence, the following hold within $R(j_1,\dots,j_s;a_1,\dots,a_s)$: 
	\begin{enumerate}
		\item there is $0$ in the column corresponding to $\mathbf{b}_\gamma$;
		\item  for each  $\delta \in \left( \{v_2, v_2+1,\dots, v_{i_0+1}\} \setminus Q \right) \cup \{v_{j_1+1},\dots,v_{j_s+1}\}$, let $C(j_1,\dots,j_s;a_1,\dots,a_s;\\ \delta)$ be the column of $(A \vert \mathbf{b}_\gamma)$ corresponding to the element of $\mathcal{B}_{i_0}$ of the form $$[v_2,\dots,\delta,\dots,v_{j_1},\dots,a_1,\dots, \widehat{v_{j_1+1}},v_{j_1+2},\dots, v_{j_s-1},v_{j_s}, \dots,a_s,\dots,\widehat{v_{j_s+1}},\dots,v_{i_0+1}].$$ The following value appears in $C(j_1,\dots,j_s;a_1,\dots,a_s;\delta)$:
		\begin{itemize}
			\item $-1$, if $v_{j_h+2} < \delta <a_{h+1}$ for some $h=0,\dots,s$ (where $j_0:=0$ and $a_{s+1}:=v_{i_0+1})$;
			\item $1$, if $a_h < \delta \leq v_{j_h+1}$ for some $h=1,\dots,s$;
		\end{itemize}
		\item All other entries are zero.
	\end{enumerate}

For example, in Table \ref{Table 2}, we present the relevant rows and columns utilized in the reduction of $R$ when considering the second squarefree power of the $3$-path ideal of $P_{14}$. Given that $i_0 = 4$, it follows that $s \in \{1,2\}$; specifically, the first four rows correspond to $s = 1$, while the last four rows to $s = 2$.

\begin{table}[h!]
	\footnotesize
	\scalebox{0.71}{
	\begin{tabular}{c|c|c|c|c|c|c|c|c|c|c}
		& 7 9 11 12 13 & 7 8 9 11 13 & 7 8 9 11 12 & 7 8 11 12 13 & 7 8 11 12 14 & 7 8 11 13 14 & 7 8 9 11 14 & 7 9 11 12 14 & 7 9 11 13 14 & 7 9 11 12 13 \\ \hline
		\rowcolor{gray!10}
		7 9 11 12  & 1           &  0           & -1          & 0             & 0             &    0          &   0          & 1            &    0          & 1            \\ \hline
		7 8 11 13  &   0           & 1           &   0          & -1           &   0           & 1            &   0          &   0           &     0         &    0          \\ \hline
		\rowcolor{gray!10}
		7 8 11 12  &       0       &       0      & 1           & 1            & 1            &      0        &     0        &       0       &   0           &   0           \\ \hline
		7 9 11 13  & -1           & -1          &      0       &     0         &    0          &   0           &      0       &   0           & 1            & -1           \\ \hline
		&              &           &             &             &              &             &            &              &              &              \\ \hline
		\rowcolor{gray!10}
		7 8 11 14  &        0      &      0       &     0        &    0          & -1           & -1           & 1           &  0            & 0             &    0          \\ \hline
		7 9 11 14  &     0         &      0       &       0      &  0            &    0          &     0         & -1          & -1           & -1           &   0           \\ \hline
		\rowcolor{gray!10}
		7 10 11 12 &    0          &      0       &     0        &    0          &        0      &      0        &    0         &   0           &      0        &       0       \\ \hline
		7 10 11 13 &      0        &     0        &    0         &     0         &   0           &        0      &   0          &         0     &     0         &    0          \\
	\end{tabular}%
}\\
%
%
%
\vspace{0.2cm}
	\scalebox{0.71}{
	\begin{tabular}{c|c|c|c|c|c|c|c|c|c|c}
		7 8 9 11 13 & 7 8 10 11 12 & 7 9 10 11 12 & 7 8 10 11 13 & 7 9 10 11 13 & 7 10 11 12 13 & 7 8 10 11 14 & 7 9 10 11 14 & 7 10 11 12 14 & 7 10 11 13 14 & $\mathbf{b}_\gamma$\\ \hline
		\rowcolor{gray!10}
		0&    0          & 1            &      0        &       0       &          0     &        0      &    0          &     0          &      0         & 0            \\ \hline
		1           &    0          &       0       & 1            &     0         &      0         &      0        &     0         &   0            &       0        & 0            \\ \hline
		\rowcolor{gray!10}
		0& 1            &      0        &     0         &       0       &        0       &      0        &     0         &     0         &     0          & 0            \\ \hline
		-1          &    0          &        0      &      0        & 1            &     0          &    0          &     0         &      0         &      0         & 0            \\ \hline
		&              &              &              &              &               &              &              &               &               &              \\ \hline
		\rowcolor{gray!10}
		0&     0         &     0         &        0      &      0        &       0        & 1            &  0            &       0        &       0        & 0            \\ \hline
		0&       0       &          0    &       0       &      0        &        0       &      0        & 1            &       0        &    0           & 0            \\ \hline
		\rowcolor{gray!10}
		0& -1           & -1           &      0        &       0       & 1             &     0         &      0        & 1             &      0         & 0            \\ \hline
		0&       0       &       0       & -1           & -1           & -1            &      0        &     0         &       0        & 1             & 0            \\
	\end{tabular}
}
\caption{Rows used for the elementary operations on $R$.}
\label{Table 2}
\end{table}

Our aim is to prove that prove that we can reduce $R$ to a row with zero entries everywhere except for a $1$ in the column corresponding to $\mathbf{b}_\gamma$. This reduction is achieved by successively adding to $R$ the rows $R(j_1,\dots,j_s; a_1,\dots,a_s)$, for $s =1,\dots,i_0/2$, following a stepwise process.

For example, the reader can easily verify that the row $R$ in Table \ref{Example row R}  can be reduced to the desired form by summing the rows from Table \ref{Table 2} to $R$. This process involves first adding the initial four rows, followed by the remaining ones, so that we cancel out $1$ and $-1$ entries in each column to obtain zeros. This example illustrates the key steps of the reduction process. We now present the general procedure. Throughout, we adopt the notation $M(U,V)$ to denote the entry in row $U$ and column $V$ of a matrix $M$.

{\bf Step 1:} For simplicity, set $M=(A \vert \mathbf{b}_\gamma)$. In $M$, we replace $R$ by $R + \sum R(j_1; a_1)$, where the sum ranges over all even $j_1 \in \{2, \dots, i_0\}$ and $a_1 \in \mathbb{N}$ such that $v_{j_1} < a_1 < v_{j_1+1}$. This results in the following:
\begin{enumerate}
	\item The $-1$ at $M(R, C_{j_1, a_1})$ is reduced to zero by adding the $1$ at $M(R(j_1; a_1), C(j_1; a_1; \delta)$), where $\delta = v_{j_1+1}$ (note that $C(j_1; a_1; \delta)$ and $C_{j_1, a_1}$ correspond to the same column in $M$). In our example, the $-1$ in $C_{4,12}$ reduces to $0$ by adding $1$ in $M(R(4;12),(4;12;14))$. 
	\item For every $\delta \in \left( \{v_2, v_2+1,\dots, v_{i_0+1}\}\setminus Q \right)$, the following hold:
	\begin{enumerate}
		\item If $a_1 < \delta < v_{j_1+1}$, the entry at $M(R, C_{j_1, a_1})$ is zero, because the $1$ at $M(R(j_1; a_1), C(j_1; a_1; \delta))$ cancels out the $-1$ at $M(R(j_1, \delta), C(j_1; a_1; \delta))$. For example, $1$ in $M(R(4;12),C(4;12;13))$ cancels out $-1$ in $M(R(4;13),C(4;12;13))$. 
		\item If $\delta < a_1$ or $v_{j_1+1} < \delta < v_{i_0+1}$, the entry at $M(R ,  C(j_1; a_1; \delta))$ is $-1$. For example, see $M(R(4;12),C(4;12;9))$ and $M(R(4;12),C(4;12;8))$.
	\end{enumerate}
\end{enumerate}
	We denote the resulting row and the resulting matrix, after applying these elementary operations, by $R_1$ and $M_1$ respectively.

{\bf Step 2:}  In the matrix $M_1$, we replace $R_1$ by $R_1 + \sum R(j_1,j_2; a_1,a_2)$, where the sum ranges over all even $j_1,j_2 \in \{2, \dots, i_0\}$ with $j_1<j_2$ and $a_1, a_2 \in \mathbb{N}$ such that $v_{j_1} < a_1 < v_{j_1+1}$ and $v_{j_2} < a_2 < v_{j_2+1}$. We obtain the following result:
	\begin{enumerate}
		\item From Step 1 the entry at $M_1(R , C(j_1; a_1; \delta))$ is $-1$, if $\delta < a_1$ (similarly, if $v_{j_1+1} < \delta < v_{i_0+1}$). Thus, there exists an $h \in \{0,\dots,s\}$ such that $v_{j_h} < \delta < v_{j_{h+1}}$, allowing us to cancel $-1$ by adding $1$ from $M_1(R(j_h,j_1;\delta,a_1) ,  C(j_h,j_1;\delta,a_1;v_{j_h+1}))$. For example, $-1$ at $M_1(R(4;12),C(4;12;9))$ is canceled by adding $1$ from $M_1(R(2,4;9,12),C(2,4;9,12;10))$.   
		\item For every $\delta \in \{v_2, v_2+1,\dots, v_{i_0+1}\}\setminus Q $, the following hold:
		\begin{enumerate}
			\item If $a_1 < \delta < v_{j_1+1}$ (or $a_2 < \delta < v_{j_2+1}$), the entry at $M_1(R ,  C(j_1, j_2;a_1,a_2;\delta))$ becomes zero, since the $1$ at $M_1(R(j_1,j_2; a_1,a_2)$, $ C(j_1,j_2; a_1,a_2;\delta))$ cancels the $-1$ at $M_1(R(j_1,j_2; \delta,a_2)$, $ C(j_1,j_2; \delta,a_2;a_1))$. In our example, $-1$ at $M_1(R(2,4;9,12)$, $C(2,4;9,12;8))$ can be reduced to zero by summing with $1$ at $M_1(R(2,4;8,12)$, $ C(2,4;8,12;9))$. 
			\item If $\delta < a_1$ (or $v_{j_1+2} < \delta < a_2$ or $v_{j_2+2} < \delta < v_{i_0+1}$), the entry at $M_{1}(R , C(j_1,j_2;a_1,a_2;\delta))$ remains $-1$. 
		\end{enumerate}
        \end{enumerate}
        After this step, we denote the updated row and matrix by $R_2$ and $M_2$ respectively.
        
{\bf Step $s$ (with $2 < s < i_0/2$):} In the matrix $M_{s-1}$ obtained at the $(s-1)$-th step, we replace the row $R_{s-1}$ by $R_{s-1} + \sum R(j_1,\dots,j_{s}; a_1,\dots,a_{s})$, where the sum ranges over all even $j_1,\dots,j_{s} \in \{2, \dots, i_0\}$ with $j_1<\dots<j_{s}$ and $a_1,\dots,a_{s} \in \mathbb{N}$ such that $v_{j_h} < a_h < v_{j_h+1}$ for $h\in [s]$. As observed in Step 2, the only non-zero entries in $R_s$ are the $-1$ located at $M_s(R_s, C(j_1, \dots, j_s; a_1, \dots, a_s; \delta))$, where either $\delta < a_1$, or $v_{j_h+2} < \delta < a_h$ for some $h \in [s-1]$, or $v_{j_s+2} < \delta < v_{i_0+1}$. 

Now, continue the procedure described above until the final step, which is given for $s=i_0/2$. In our example, the process concludes at the second step since $s = 2$, and no additional $-1$ values arise from this step, meaning that the desired result has been achieved.	
	
{\bf Step $i_0/2$:} For simplicity, let $i_0/2=\ell$. In the matrix $M_{\ell-1}$, we replace $R_{\ell-1}$ by $R_{\ell-1} + \sum R(j_1,\dots,j_{\ell}; a_1,\dots,a_{\ell})$, where the sum ranges over all even $j_1,\dots,j_{\ell} \in \{2, \dots, i_0\}$ with $j_1<\dots<j_{\ell}$ and $a_1,\dots,a_{\ell} \in \mathbb{N}$ such that $v_{j_h} < a_h < v_{j_h+1}$ for $h\in [\ell]$. The only significant case that remains to be analyzed is when $v_{j_{h-1}+2} < \delta < a_h$ for some $h\in [\ell]$ but, in this case, it suffices to proceed analogously to the argument presented in (1) of Step 2.

	This guarantees that summing to $R$ the rows $R(j_1,\dots,j_s; a_1,\dots,a_s)$, for $s=1,\dots,i_0/2$, we get a row with zero entries everywhere except for a single $1$ in the column corresponding to $\mathbf{b}_\gamma$. As a consequence, the coefficient matrix $A$ and the augmented matrix $(A \vert \mathbf{b}_\gamma)$ have different ranks. This completes the proof of our claim.

 Case (2): Now, suppose $kt+j(t+1) < n\leq kt+j(t+1)+t$, for some $j$. Let $\Delta(I_{n,t}^{[k]})$ be the simplicial complex whose facet ideal is $I_{n,t}^{[k]}$. Observe that $\nu_1(\Gamma_{n-(k-1)t,t})=\nu_1(\Gamma_{n_0-(k-1)t,t})$. It is sufficient to show that $\beta_{i_0,n_0} (I_{n,t}^{[k]})\neq 0$, where $i_0$ and $n_0$ are defined in Case 1. From \cite[Theorem 2.8]{AF} we know that
	\[
	\beta_{i_0,n_0} (I_{n,t}^{[k]})=\sum_{\substack{\Gamma \subseteq \Delta(I_{n,t}^{[k]})\\ \vert V(\Gamma)\vert =n_0}} \dim_K \tilde{H}_{i_0-1} (\Gamma^c_{V(\Gamma)}),
	\]
	where the sum is taken over the induced subcollections $\Gamma$ of $\Delta(I_{n,t}^{[k]})$ which have $n_0$ vertices. Take $\Gamma =\Delta(I_{n_0,t}^{[k]})$ and observe that $\Delta(I_{n_0,t}^{[k]})$, defined as in Case (1), is an induced subcollection of $\Delta(I_{n,t}^{[k]})$ having $n_0$ vertices. From Case (1), we know that $\dim_K \tilde{H}_{i_0-1} (\Gamma^c_{V(\Gamma)})\neq 0$, so $\beta_{i_0,n_0} (I_{n,t}^{[k]})\neq 0$, which implies $(k-1)t+(t-1)\nu_1 (\Gamma_{n-(k-1)t,t}) \leq \reg (I_{n,t}^{[k]})$.
 
 In conclusion, we get the desired lower bound.
    \end{proof}
    
We illustrate the proof of above theorem in the following example.

\begin{Example}\em
    Referring to the notation of the proof of Theorem \ref{thm: lower bound}, in Table \ref{Table1} we display the regularity of the third squarefree power of the $3$-path ideal of $P_n$ for $9\leq n\leq 20$. The elements of set $Q$ are displayed by the hollow circles. 
\begin{footnotesize}
    
  \begin{table}[h]
		\centering
		\renewcommand\arraystretch{1.3}{	\begin{tabular}{c|Hc|c}
  
		$n$ & $n_0$ & $\reg(I_{n,3}^{[3]})$  & \\
            \hline

     9 & 9  & 9 & \includegraphics[scale=0.71]{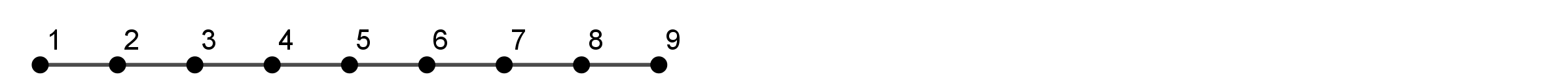}\\
		\hline
    
     10 & 9  & 9 & \includegraphics[scale=0.71]{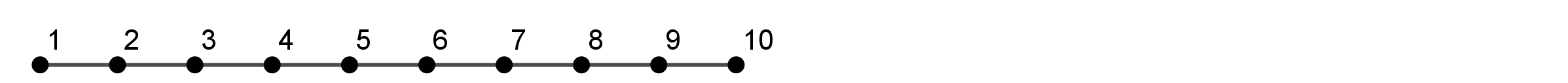}\\
		\hline

     11 & 9  & 9 & \includegraphics[scale=0.71]{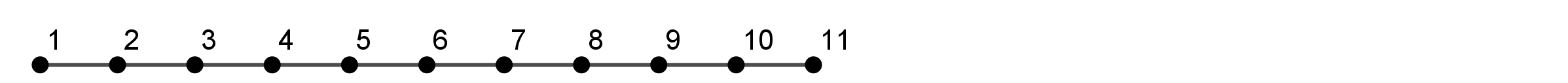}\\
		\hline

    12 & 9  & 9 & \includegraphics[scale=0.71]{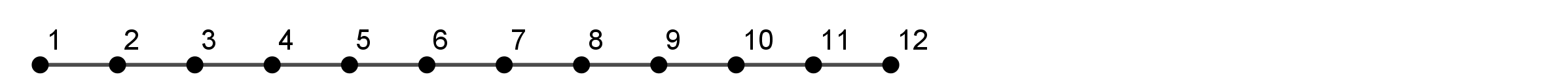}\\
		\hline

    13 & 13  & 11 & \includegraphics[scale=0.71]{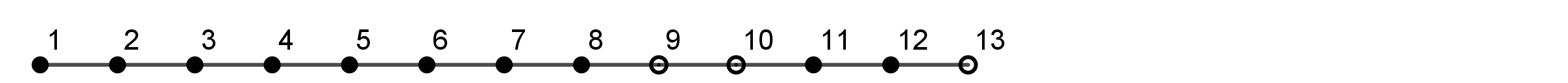}\\
		\hline
    
    14 & 13  & 11 & \includegraphics[scale=0.71]{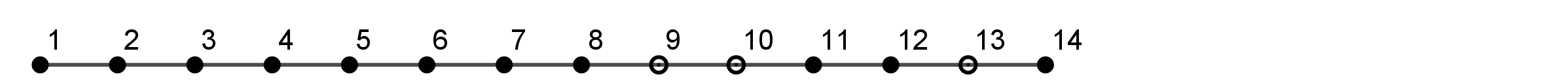}\\
		\hline
   
    15 & 13  & 11& \includegraphics[scale=0.71]{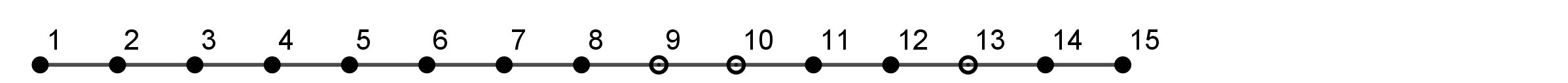}\\
		\hline
    
    16 & 13  & 11 & \includegraphics[scale=0.71]{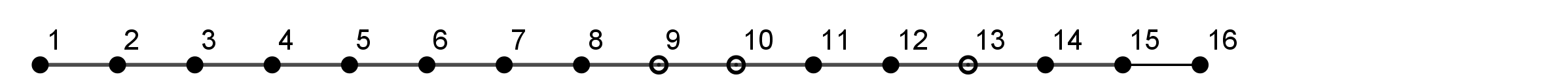}\\
		\hline
    17 & 17  & 13 & \includegraphics[scale=0.71]{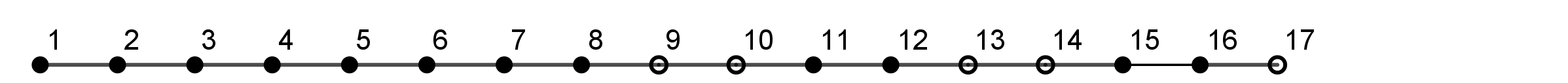}\\
		\hline
   18 & 17  & 13 & \includegraphics[scale=0.71]{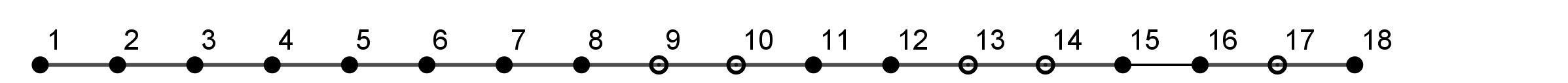}\\
		\hline
  
  19 & 17  & 13 & \includegraphics[scale=0.71]{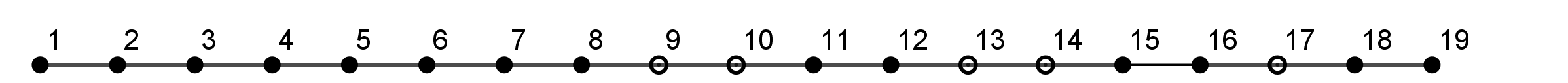}\\
		\hline
 
 20 & 17  & 13 & \includegraphics[scale=0.71]{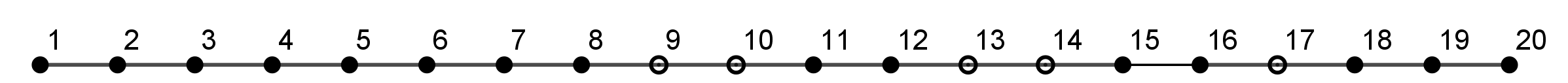}\\
		
		\end{tabular}}
		\caption{Regularity and elements of $Q$ for $I_{n,3}^{[3]}$}
		\label{Table1}
	\end{table}
 \end{footnotesize}

For instance, take $n=17$. Since $17=kt+2(t+1)$ for $k=t=3$, we are in Case (1) of the proof of above theorem. Here we have $i_0=4$, and  $Q=\{9,10,13,14,17\}$, and $\gamma=[10,13,14,17]-[9,13,14,17]+[9,10,14,17]-[9,10,13,17]+[9,10,13,14]$. Moreover,
 \begin{itemize}
     \item $[10,13,14,17]\in C_{3}((\Delta(I_{17,3}^{[3]}))^c)$ because $\{1,\dots,9\}$ is a facet of $\Delta(I_{17,3}^{[3]})$;
     \item $[9,13,14,17]\in C_{3}((\Delta(I_{17,3}^{[3]}))^c)$ because $\{1,\dots,6\}\cup\{10,11,12\}$ is a facet of $\Delta(I_{17,3}^{[3]})$;
     \item $[9,10,14,17]\in C_{3}((\Delta(I_{17,3}^{[3]}))^c)$ because $\{1,\dots,6\}\cup\{11,12,13\}$ is a facet of $\Delta(I_{17,3}^{[3]})$; 
     \item $[9,10,13,17]\in C_{3}((\Delta(I_{17,3}^{[3]}))^c)$ because $\{1,\dots,6\}\cup\{14,15,16\}$ is a facet of $\Delta(I_{17,3}^{[3]})$; 
     \item $[9,10,13,14]\in C_{3}((\Delta(I_{17,3}^{[3]}))^c)$ because $\{1,\dots,6\}\cup\{15,16,17\}$ is a facet of $\Delta(I_{17,3}^{[3]})$.
 \end{itemize}
 It is easy to see that $\delta_{3}(\gamma)=0$, so $\gamma\in  \ker\delta_3$. Moreover, $[9,10,13,14,17]$ does not belong to $C_{4}((\Delta(I_{17,3}^{[3]}))^c)$ because there is no $3$-matching in the induced subgraph of $P_{17}$ on $\{1,\dots,8\}\cup\{11,12\}\cup \{15,16\}$. Finally, it can be verified by using \texttt{Macaulay2} (\cite{M2}, \cite{HSZ}) that $\gamma$ is not in $\Im{\delta_{4}}$, since the rank of the representative matrix of $\delta_{4}$ differs from that of the corresponding augmented matrix. Therefore, $\tilde{H}_{3} ( (\Delta(I_{17,3}^{[3]}))^c; K)\neq 0$ and so $\beta_{4,17} (I_{17,3}^{[3]})\neq 0$.
 
 Now, let $n=19$. We know that \[
	\beta_{4,17} (I_{19,3}^{[3]})=\sum_{\substack{\Gamma \subseteq \Delta(I_{19,3}^{[3]})\\ \vert V(\Gamma)\vert =17}} \dim_K \tilde{H}_{3} (\Gamma^c_{V(\Gamma)}),
	\]
	where the sum is taken over the induced subcollections $\Gamma$ of $\Delta(I_{19,3}^{[3]})$ which have $17$ vertices. Observe that $\Delta(I_{17,3}^{[3]})$ is an induced subcollection of $\Delta(I_{19,3}^{[3]})$ on $\{1,\dots,17\}$ vertices and $\tilde{H}_{3} ( (\Delta(I_{17,3}^{[3]}))^c; K)\neq 0$ as discussed before. Therefore $\beta_{4,17} (I_{19,3}^{[3]})\neq 0$.
 
\end{Example}

  Combining Theorem~\ref{thm: upper bound} and Theorem~\ref{thm: lower bound}, we obtain the following nice combinatorial description of the regularity of squarefree powers of $t$-path ideals of path graphs.
  
     \begin{Theorem}\label{thm: equality regularity t-path ideals of path}
    
    Let $\Gamma_{n,t}$ be the $t$-path simplicial tree of a path graph $P_n$ and $I_{n,t}= I(\Gamma_{n,t})$. Then for any $1\leq k+1 \leq \nu(\Gamma_{n,t})$, we  have
		\[
		\reg\left( \frac{R}{I_{n,t}^{[k+1]}}\right)= kt+(t-1)\nu_1 (\Gamma_{n-kt,t})=kt+\reg  \left( \frac{R}{I_{n-kt,t}}\right).
		\]
	\end{Theorem}
In \cite[Theorem 4.7]{KK1}, the authors provided a formula for the regularity of the powers of $t$-path ideals of broom graphs, which also holds for path graphs. According to their formula,
	\[
		\reg\left( \frac{R}{I_{n,t}^{k+1}}\right)= kt+\reg  \left( \frac{R}{I_{n,t}}\right).
		\]
By comparing the above formula with the regularity of the squarefree powers of $I_{n,t}$ computed in Theorem~\ref{thm: equality regularity t-path ideals of path}, we observe that for any $ 2\leq k+1 \leq \nu(\Gamma_{n,t})$, we have  $\reg\left( \frac{R}{I_{n,t}^{[k+1]}}\right)= \reg\left( \frac{R}{I_{n,t}^{k+1}}\right)$ if and only if $\reg  \left( \frac{R}{I_{n-kt,t}}\right) = \reg  \left( \frac{R}{I_{n,t}}\right)$. Using Lemma~\ref{lem:regofpath}, we conclude that the last equality holds if and only if $k=1$ and 
$n=m(t+1)+t-1$ for some positive integer $m$. 


    \section{Towards Future Research}\label{Section: Towards future works}

In this section we propose some open questions, which can inspire new  work related to squarefree powers of monomial ideals.
\medskip

 1) We know that $I(\Delta)^{[\nu(\Delta)]}$ has a linear resolution if $\Delta$ is a simplicial tree with the intersection property (see  Theorem~\ref{thm: simplicial tree with i.p. has linear quotients}) or $\Delta$ is a $t$-path simplicial tree of a broom graph (see Theorem~\ref{thm:broomgraph-linearquotient}). However, given an arbitrary simplicial tree $\Delta$, Example \ref{exa: 3-path of a rooted tree with no linear resolution} shows that $I(\Delta)^{[\nu(\Delta)]}$ need not have a linear resolution. It is interesting to characterize those simplicial trees $\Delta$ for which $I(\Delta)^{[\nu(\Delta)]}$ has a linear resolution.
 \medskip
 
     2) Let $\Delta$ be a pure simplicial tree. From Theorem \ref{theo:linearly related}, we know that if $I(\Delta)^{[k]}$ is linearly related then $I(\Delta)^{[k+1]}$ is also linearly related. Does this statement hold for the linearity of the resolution, that is, if $I(\Delta)^{[k]}$ has a linear resolution then does $I(\Delta)^{[k+1]}$ has also a linear resolution?  
     \medskip
     
    3) In Theorem \ref{thm: equality regularity t-path ideals of path}, we provide a closed formula for $\reg(R/I_{n,t}^{[k]})$, where $I_{n,t}$ is the $t$-path ideal of the path graph $P_n$. In general, it seems very difficult to derive such a formula for the squarefree powers of an arbitrary simplicial tree. It could be interesting to establish at least an upper or lower bound for $\reg(R/I(\Delta))$, where $\Delta$ is a simplicial tree.
    
In a more general context, a lower bound for certain squarefree powers can be obtained as an easy generalization of \cite[Theorem 2.1]{EHHM}, whose proof is provided below for the sake of completeness.

	\begin{Proposition}	\label{prop:lowerbound}
		Let $\Delta$ be a pure simplicial complex of dimension $d-1$. Then \[ k-1+(d-1)\nu_1(\Delta) \leq 
		\reg\left(\frac{R}{I(\Delta)^{[k]}} \right)\] 
		for all $1 \leq k \leq \nu_1(\Delta)$. 
	\end{Proposition}
	\begin{proof}
	Denote by $r$ the induced matching number of $\Delta$, for simplicity. Let $\Delta'$ be a subcollection of $\Delta$ with $r$ facets $\{F_1,\dots, F_r\}$ which is an induced matching of $\Delta$. Since $\reg (I(\Delta')^{[k]})\leq \reg (I(\Delta)^{[k]})$ from \cite[Lemma 4.4]{HHZ}, it is enough to prove that $k+(d-1)r \leq \reg (I(\Delta')^{[k]})$, in particular that $\beta_{r-k,rd}(I(\Delta'))\neq 0$. Consider the ideal $J=(z_1,\dots,z_r)$ in a new polynomial ring $R=K[z_1,\dots,z_r]$. Since $J^{[k]}$ is a squarefree strongly stable ideal in $R$, we have $\beta_{r-k,r}(J^{[k]})\neq 0$ from \cite[Theorem 7.4.1]{HH1}. Now, let $f_j=\prod_{i\in F_j}x_i$. Since $\Delta$ is pure of dimension $d-1$ and $\{F_1,\dots, F_r\}$ is an induced matching of $\Delta$, then $f_1,\dots,f_r$ is a regular sequence in $S=K[x_i:i\in \cup_{j=1}^rF_j]$ and $\deg(f_i)=t$ for all $i=1,\dots,r$. Set $I=I(\Delta')=(f_1,\dots,f_r)$ and the map $\phi:R\rightarrow S$ with $\phi(z_j)=f_j$ for all $j=1,\dots,r$. As explained in \cite[Theorem 2.1]{EHHM}, the $i$-th free module in the minimal free resolution of $I^{[k]}$ is given by $S(-dk-di)^{\beta_i(I^{[k]})}$ and $\beta_{i,j}(J^{[k]})=\beta_{i,tj}(I^{[k]})$. Then $\beta_{r-k,rd}(I)\neq 0$ because $\beta_{r-k,r}(J^{[k]})\neq 0$.
	\end{proof}

    We recall that, for a simple graph $G$, it is expected that 
    $$k + \nu_1(G) \leq \reg(I(G)^{[k]}) \leq k + \nu(G)$$ 
    for $1\leq k\leq \nu(G)$, as discussed in \cite[page 3]{EHHM}. Based on the results of this paper, we present the following conjecture. 
    
    \begin{Conjecture}
         Let $\Delta$ be a simplicial tree of dimension $d-1$. Then 
         \[ k-1+(d-1)\nu_1(\Delta) \leq 
		\reg\left(\frac{R}{I(\Delta)^{[k]}} \right)\leq k-1+(d-1)\nu(\Delta)
        \] 
		for all $1 \leq k \leq \nu(\Delta)$. 
    \end{Conjecture}

    4) In \cite{AF2}, a combinatorial description of all graded Betti numbers of $t$-path ideals of path graphs and cycle graphs is given. It would be of interest to give such a description for the squarefree powers of $t$-path ideals of path graphs and cycle graphs.\\

		{\bf Declaration of competing interest}
  
The authors declare that they have no known competing financial interests or personal
relationships that could have appeared to influence the work reported in this paper.

{\bf Data availability}

No data was used for the research described in the article.

	\end{document}